\documentclass{amsart}

\usepackage{amssymb}
\usepackage{hyperref}
\hypersetup{
    citecolor=blue,
    filecolor=blue,
    linkcolor=blue,
    urlcolor=blue
}

\newtheorem{thm}{Theorem}[section]

\newtheorem{lem}[thm]{Lemma}
\newtheorem{cor}[thm]{Corollary}

\theoremstyle{definition}
\newtheorem{definition}[thm]{Definition}

\theoremstyle{remark}
\newtheorem{remark}[thm]{Remark}

\numberwithin{equation}{section}

\renewcommand{\phi}{\varphi}
\renewcommand{\subset}{\subseteq}		

\title[Free boundary regularity in obstacle problems]{Free boundary regularity in obstacle problems with a degenerate forcing term}
\author{Yong Liu}
\thanks{The paper is supported by Hunan Provincial Innovation Foundation For Postgraduate [CX20230405].   }
\address{School of Mathematics, Hunan University, Changsha, 410082, China}
\email{yongliu93@hnu.edu.cn}

\begin{document}

\maketitle

\begin{abstract}
In this paper, we consider the properties of a special free boundary point in the following obstacle problem
 \begin{equation*}
\Delta u=f(x)\chi_{\{u>0\}} \textup{ in $B_1$},
\end{equation*}
where  $f(x)=|x|$ is a degenerate forcing term and $B_1\subset \mathbb{R}^2$.

The key challenge stems from the degeneracy of $f(x)$, which leads to a  more complex structure of the free boundary compared to the classical setting.  To analyze it, we introduce the epiperimetric inequality developed by Weiss (Invent Math 138:23-50, 1999). Although this powerful tool was firstly introduced for the classical obstacle problem characterized by $\inf_{B_1} f(x)>0$ (see \eqref{f2}), it also proves effective in our degenerate setting.
This allows us to first obtain the decay rate of the Weiss energy for all blow-ups at the origin, which in turn implies the uniqueness of the blow-up profiles. With this uniqueness established, we then prove a very weak directional monotonicity
properties satisfied by the solutions. This finally yields the regularity of the free boundary at the origin if the origin is a regular point.
\end{abstract}

\tableofcontents

\section{Introduction}

The classical obstacle problem describes the height of an elastic membrane being
pushed towards an impenetrable obstacle. Mathematically, it can be written as the minimization problem
\begin{equation*}
\textrm{min}\quad \frac{1}{2}\int_{\Omega}|\nabla v|^2dx \quad \textrm{among all functions}\quad v\ge\varphi,
\end{equation*}
where the minimization is subject to boundary conditions $v|_{\partial \Omega}=g$, and the ``obstacle" $\varphi$ meets  $\varphi|_{\partial \Omega}<g$. If we use $u$ to represent the height of the minimizer minus the height of the obstacle, then one can see that
\begin{equation}\label{f1}
\left\{
\begin{split}
&\Delta u=f(x)\chi_{\{u>0\}} \textup{ in $B_1$},\\
&u\ge0 \textup{ in $B_1$},\\
\end{split}
\right.
\end{equation}
where $f(x):=-\Delta\varphi$. For the sake of studying the regularity of the free boundary $\partial \{u=0\}$, as already observed in \cite{Caf1,Caf2}, it is necessary to impose the nondegeneracy condition
\begin{equation}\label{f2}
0<\lambda\le f(x).
\end{equation}
Here, $\lambda$ is a constant. Since this type of problem and its variants appear extensively in fluid mechanics, elasticity, probability, finance, biology, and industry, etc., many researchers have focused their efforts on this field over the past several decades (see \cite{Dan,Pet,Xa}).

As shown by Br\'{e}zis and Kinderlehrer \cite{Br}, the problem has a unique solution which is $C^{1,1}_{loc}$. In the dimension $n=2$, Sakai \cite{Sa1,Sa2} gave some structural results for the free boundary provided $f$ is real analytic. For the higher dimension, Caffarelli's seminal work \cite{Caf1,Caf2} implies that the free boundary $\partial\{u = 0\}$ splits into the regular points which behaves like a half-space solution $\frac{1}{2}\max\{x\cdot e,0 \}^2$ with some $e\in \mathbb{S}^{n-1}$ and singular points which behaves like a nonnegative homogeneous quadratic polynomial $\frac{1}{2}x\cdot Ax$ with some positive semidefinite matrices $A$. Meanwhile, he also proved that near the regular points, the free boundary locally is a $C^1$ hypersurface (and then Kinderlehrer and Nirenberg showed that the free boundary is smooth and analytic \cite{Kin2}). Conversely, the set of singular points can be covered  locally  by a $C^1$ hypersurface. However,
Schaeffer showed that the free boundary can form cusps near the singular point, which implies that
the singular set can be rather wild \cite{Sch}. It is worth noting that Caffarelli's robust method \cite{Caf1,Caf2}
has been extended to other fields such as the obstacle problem for integro-differential operators \cite{Bar,Caf3}, the thin obstacle
problem \cite{Ath,Savin1}, and fully nonlinear obstacle problem \cite{Lee,Xa1}, etc.. However, Weiss
provided a new perspective to study the free boundary. He introduced an epiperimetric inequality for
 the Weiss energy of the classical obstacle problem in the celebrated work \cite{Wei1}. In particular, he improved the $C^1$ regularity to
$C^{1,\alpha}$ at the regular point in two dimensions. This robust method also has been extended to many different settings \cite{Ander,Colo0,Colo1,Eng,Spol11}.
Recently, in a pioneering work, by introduced the Almgren frequency formula for classical obstacle problems, Figalli and Serra  proved
a more refined structure of the singular set \cite{Fig1}. By developing a novel approach, Savin and Yu showed the similar results for the fully nonlinear obstacle problem \cite{Savin2,Savin3}.
For more results on singularities, we refer to \cite{Caff12,Gar,Mon1,Xa}.

In recent years, the obstacle problem with a degenerate force term has increasingly attracted considerable attention.
Under this context, we just have
\begin{equation}\label{f3}
0\le f(x),
\end{equation}
which implies that $f(x)$ may degenerate at certain points. Apart from its various industrial applications,
mathematically, as a generalization of the classical obstacle problem, it is also very interesting to study the regularity
of the free boundary. However,  as we previously indicated, once the force term $f(x)$ becomes degenerate,
the analysis of the problem becomes considerably more challenging in contrast to the classical obstacle problem.
Significantly, Yeressian in \cite{Yere1} first provided a preliminary attempt. He established the non-degeneracy
 property and the optimal growth estimate for a class of functions satisfying
\begin{equation}\label{f4}
\lambda|(x_1,..,x_p)|^\alpha\le f(x)\le \Lambda|(x_1,..,x_p)|^\alpha,
\end{equation}
where $\lambda$, $\Lambda$, $\alpha$ and $p\le n$ are positive constants. Furthermore,  Yeressian \cite{Yere2}  considered the following
problem in two dimensions
\begin{equation}\label{f5}
\left\{
\begin{split}
&\Delta u=|x_1|\chi_{\{u>0\}} \textup{ in $B_1$},\\
&u\ge0 \textup{ in $B_1$}.\\
\end{split}
\right.
\end{equation}
On one hand, he proved the free boundary is locally a graph in a neighborhood of the point with the lowest Weiss balanced energy. On the other hand,  near to a degenerate point, he showed that the free boundary is not $C^{1,\alpha}$ for any $0<\alpha<1$ under some  suitable assumptions.

While these results are exciting, research on obstacle problems governed by degenerate forcing terms remains very limited. And our understanding of how the degeneration mechanism affects the free boundary is still unclear. Hence, inspired  by the above work, in this paper, we study the following problem
\begin{equation}\label{f6}
\left\{
\begin{split}
&\Delta u=|x|\chi_{\{u>0\}} \textup{ in $B_1$},\\
&u\ge0 \textup{ in $B_1$}.\\
\end{split}
\right.
\end{equation}
In contrast to classical obstacle problem \eqref{f1}, although the function $f(x)=|x|$ degenerates only at the origin in our setting, the free boundaries exhibit significant differences. In the subsequent analysis, we will also observe certain novel properties of \eqref{f6} that
deviate from their counterparts in the classical obstacle problem.

From a geometric perspective,  both the classical obstacle problem and its many variants share a important feature: they admit a classical Caffarelli's dichotomy theorem \cite{Caf1,Caf2}. Furthermore, Weiss \cite{Wei1} provided a novel interpretation from an energy perspective, stating that free boundary points possess two levels of Weiss energy: all regular points share the lowest energy level, while singular points share a higher energy level. Within the same energy level, despite the potential variations in the geometric characteristics of the free boundary, their values of Weiss energy remain equal. However, from an energy perspective, our setting is considerably more complex. Even when restricting our analysis to two dimensions, additional energy levels emerge. Moreover, one of the core challenges in free boundary problems is proving the uniqueness of blow-up limits. In our problem, when the blow-up limit is no longer a quadratic homogeneous polynomial, or even no longer a polynomial, it brings new and significant challenges to the study of free boundaries.

Before stating our main theorem, we first introduce the Weiss energy. Suppose $x_0\in \partial \{ u=0 \}$, then we can easily find that the Weiss energy for \eqref{f6}  is given by
\begin{equation}\label{f7}
\Phi(r,u)=\frac{1}{r}\int_{B_r(x_0)}\left(|\nabla u|^2+2|x| u^+\right)dx-\frac{3}{r^7}\int_{\partial B_r(x_0)} u^2d\mathcal{H}.
\end{equation}
If no confusion is caused, we may denote $\Phi(r,u)$ by $\Phi(r)$.   In order to obtain a  decay estimate of the energy, we consider the following  energy functional \begin{equation}\label{int1}
M(v):=\int_{B_1}\left(|\nabla v|^2+2|x| v^+ \right)dx-3\int_{\partial B_1} v^2d\mathcal{H},
\end{equation}
among all $v\in W^{1,2}(B_1)$. First, we state an important property of $M(v)$.
\begin{thm} [Epiperimetric inequality] \label{TH1}
There exist $\kappa\in(0,1)$ and $\delta\in(0,1)$ such that the following holds for every  non-negative function $c\in W^{1,2}(B_1)$ that is homogeneous of degree 3, if
\begin{equation*}
||c-h_{\theta_1} ||_{W^{1,2}(B_1)}\le \delta
\end{equation*}
for some $h_{\theta_1}\in\mathbb{H}$ (defined in Definition \ref{Def3}), then there exists a function $v\in W^{1,2}(B_1)$ satisfying $v=c$ on $\partial B_1$  and
\begin{equation}\label{cc3}
M(v)\le (1-\kappa)M(c)+\kappa M(h_{\theta_1}),
\end{equation}
\end{thm}

\begin{remark}\label{Th1.1}
We generalize the Epiperimetric inequality from the classical obstacle problem to the obstacle problem with a degenerate forcing term.
\end{remark}

\begin{thm}\label{TH1.2}
All nontrivial homogeneous global solutions fall exclusively into the following categories:
\begin{equation}\label{ff01}
 h_{\theta_1}(x)=\left\{ \frac{|x|^3}{9}-\frac{\cos3\theta_1}{9}({x_1}^3-3x_1{x_2}^2 )- \frac{\sin3\theta_1}{9}( 3{x_1}^2x_2-{x_2}^3  ) \right\}\chi_{\{ \theta_1<\theta<\theta_2 \}},
\end{equation}
and
\begin{equation}\label{ff02}
\begin{split}
h_{\widetilde{\theta_1}}(x)&=\left\{ \frac{|x|^3}{9}-\frac{\cos3\theta_1}{9}({x_1}^3-3x_1{x_2}^2 )- \frac{\sin3\theta_1}{9}( 3{x_1}^2x_2-{x_2}^3  ) \right\}\chi_{\{ \theta_1<\theta<\theta_2 \}},\\
&\quad+\left\{ \frac{|x|^3}{9}-\frac{\cos3\theta_1}{9}({x_1}^3-3x_1{x_2}^2 )- \frac{\sin3\theta_1}{9}( 3{x_1}^2x_2-{x_2}^3  ) \right\}\chi_{\{ \theta_2+\varsigma<\theta<\theta_2+\frac{2\pi}{3}+\varsigma \}},
\end{split}
\end{equation}
and
\begin{equation}\label{ff03}
 \begin{split}
h_{\underline{\theta_1}}(x)&=\left\{ \frac{|x|^3}{9}-\frac{\cos3\theta_1}{9}({x_1}^3-3x_1{x_2}^2 )- \frac{\sin3\theta_1}{9}( 3{x_1}^2x_2-{x_2}^3  ) \right\}\chi_{\{ \theta_1<\theta<\theta_2 \}},\\
&\quad+\left\{ \frac{|x|^3}{9}-\frac{\cos3\theta_1}{9}({x_1}^3-3x_1{x_2}^2 )- \frac{\sin3\theta_1}{9}( 3{x_1}^2x_2-{x_2}^3  ) \right\}\chi_{\{ \theta_2<\theta<\theta_2+\frac{2\pi}{3} \}}\\
&\quad+\left\{ \frac{|x|^3}{9}-\frac{\cos3\theta_1}{9}({x_1}^3-3x_1{x_2}^2 )- \frac{\sin3\theta_1}{9}( 3{x_1}^2x_2-{x_2}^3  ) \right\}\chi_{\{ \theta_2+\frac{2\pi}{3}<\theta<\theta_2+\frac{4\pi}{3} \}},
\end{split}
\end{equation}
as well as
\begin{equation}\label{ff04}
 h_{a,b}(x)=\frac{|x|^3}{9}+a({x_1}^3-3x_1{x_2}^2 )+b( 3{x_1}^2x_2-{x_2}^3  ),
\end{equation}
where $\theta_1\in[0,2\pi)$, $\theta_2=\theta_1+\frac{2\pi}{3}$ and $\varsigma\in[0,\frac{2\pi}{3}]$ as well as  constants $a,b$ satisfy $\sqrt{a^2+b^2}<\frac{1}{9}$. Here, in order to simplify the  notational representation of the positivity set, we use $\{ \theta_1<\theta<\theta_2 \}$ to denote the positive set  $\{x: u(x)>0, \, and \, the \, corresponding \, polar \, angle \,\ \theta\in ( \theta_1,\theta_2)\}$.

\end{thm}

\begin{remark}\label{R1.31} In the classical obstacle problem, the blow-up limit is a non-negative homogeneous quadratic polynomial-specifically, a half-space solution or a positive semi-definite quadratic form. Classical theory tells us that quadratic forms possess a well-behaved algebraic structure; as demonstrated in the seminal work of Caffarelli \cite{Caf1,Caf2} and  Figalli-Serra \cite{Fig1}, one can discuss related favorable properties across different dimensions. In our setting, however, the blow-up limit corresponds to a nonnegative third-order homogeneous global function. In practice, such function exhibit more complex geometric characteristics and lack convexity. These features introduce substantial difficulties in understanding the free boundary in this class of problems.
\end{remark}

\begin{thm}[Energy decay and uniqueness]\label{TH1.2un}  Let $0$ be a regular free boundary point, and assume the epiperimetric inequality holds with $\kappa\in(0,1)$ for every homogeneous function $c_r(x):=|x|^3u_r(\frac{x}{|x|})$ of degree $3$ provided $0<r\le r_0<1$. Let $u_0$ denote an arbitrary blowup profile of $u_r$ at $0$. Then we have
\begin{equation}\label{uni1}
|\Phi(r)-\Phi(0)|\le |\Phi(r_0)-\Phi(0)|\left( \frac{r}{r_0} \right)^{\frac{6\kappa}{1-\kappa}},
\end{equation}
and there is a constant $C(\kappa)>0$ such that
\begin{equation}\label{uni2}
\int_{\partial B_1}\left| \frac{u(rx)}{r^3}-u_0\right|d\mathcal{H}\le C(\kappa) r^{\frac{3\kappa}{1-\kappa}}
\end{equation}
for $r\in(0,r_0/2)$, and hence $u_0$ is the unique blowup profile at $0$. Here $u_r$ is defined in Definition \ref{Def2}.
\end{thm}

Now, we give the important result regarding  the free boundary as a graph near regular points.
\begin{thm}\label{TH1.4}
Let $u$ be a solution to \eqref{f6} and suppose $0$ is a regular point. Then there exist constants $C, \varrho>0$ and a function $g$ such that, after a suitable rotation of coordinate axes,
\begin{equation}\label{f8}
\Gamma(u)\cap B_{\varrho}\cap \{|x_1|< \frac{\varrho}{4}\}=\{(x_1,g(x_1)):|x_1|< \frac{\varrho}{4} \},
\end{equation}
and
\begin{equation*}
|g(x_1)-\frac{\sqrt{3}}{3}x_1|\le C||u_{8|x_1|-u^*} ||^{\frac{1}{3}}_{L^\infty (B_1)}|x_1|.
\end{equation*}
as well as
\begin{equation*}
|g'(x_1)-\frac{\sqrt{3}}{3}|\le C||u_{8|x_1|-u^*} ||^{\frac{1}{3}}_{L^\infty (B_1)}|x_1|.
\end{equation*}
where
\begin{equation}\label{f9}
 g\in C^0((-\frac{\varrho}{4},\frac{\varrho}{4}))\cap C^1 ((-\frac{\varrho}{4},\frac{\varrho}{4})\verb+\+\{0\}).
\end{equation}
\end{thm}

\begin{remark}\label{R1.41} This theorem reveals that although the free boundary remains Lipschitz near the regular point, it cannot be geometrically flat. This contrasts sharply with the key geometric property in the classical obstacle problem, where for usual regular points, ``flatness implies Lipschitz regularity''-a fundamentally new property exhibited by the regular points of \eqref{f6}. It is noteworthy that at regular points, the free boundary fails to be $C^1$. Actually, the free boundary  has a corner at the regular point. This behavior may represent a distinctive feature of the obstacle problem with a degenerate forcing term.
\end{remark}

\addtocontents{toc}{\protect\setcounter{tocdepth}{0}}
\section*{Acknowledgements}
I would like to thank Professor Hui Yu (National University of Singapore) for many helpful conversations and valuable suggestions.
\addtocontents{toc}{\protect\setcounter{tocdepth}{1}}


\section{Preliminaries }		\label{sect Preliminaries}

In this chapter, we will give some basic properties for our problem. \\

By abuse of notation, we given following definition of regular points and singular points in our problem:
\begin{definition}\label{Def1}
Let $u$ solve \eqref{f6}, and $x_0\in \{f=0\}\cap\partial\{u=0\}$, then we call

$(\romannumeral 1)$  $x_0$ is \textbf{regular point} if  $\Phi(0+,u)=\frac{\pi}{81}$,

$(\romannumeral 2)$ $x_0$ is \textbf{singular point} if $\Phi(0+,u)>\frac{\pi}{81}$.
\end{definition}
It should be emphasized that this paper focuses primarily on the properties of free boundary points situated where $f(x)$ is degenerate. Those occurring where $f(x) > 0$, however, exhibit behavior analogous to the free boundary points in the classical obstacle problem. To prevent confusion, we will hereafter refer to the free boundary points in the classical obstacle problem-as well as those in our setting that do not lie at degenerate points of $f(x)$ -as \textbf{usual regular points} or \textbf{usual singular points}.

\begin{definition}\label{Def2} (Rescaling and blowups)
Let $u$ solve \eqref{f6}, and suppose $x_0\in B_1$. We define the rescaling of $u$ at $x_0$  as
\begin{equation*}
u_{x_0,r}(x):=\frac{u(x_0+rx)}{r^3}.
\end{equation*}
If there exists a sequence $r_j\rightarrow0$ such that $u_{x_0,r_j}(x)\rightarrow u_{x_0,0}(x)$, then we call the limit of sequence $\{u_{x_0,r_j}\}$ as a blowup limit of $u$ at $x_0$.
If $x_0=0$, we denote $u_{x_0,0}$ by $u_0$.
\end{definition}

To facilitate the characterization of the epiperimetric inequality, we first introduce the following definition:
\begin{definition}\label{Def3}
We give a class of special homogeneous function of degree $3$ as follows:

\begin{equation}\label{DH1}
\mathbb{H}:= \left\{h_{\theta_1}=\left( \frac{|x|^3}{9}-\frac{\cos3\theta_1}{9}\Big({x_1}^3-3x_1{x_2}^2 \Big)- \frac{\sin3\theta_1}{9}\Big( 3{x_1}^2x_2-{x_2}^3  \Big) \right) \chi_{\{ \theta_1<\theta<\theta_2 \}}\right\},
\end{equation}
where $\theta_2=\theta_1+\frac{2\pi}{3}$ with $\theta_1\in[0,2\pi)$.
\end{definition}

The optimal nondegeneracy and  the optimal growth of solutions are crucial properties, which imply that the solutions neither decay too rapidly nor grows too quickly near free boundary points. In our setting, these key features were established by Yeressian \cite{Yere1}.

\begin{lem}\label{dege} There exists a $C>0$ such that if
\begin{equation*}
f(x)=|x|  \quad\textrm{in}\quad B_1,
\end{equation*}
then for $x_0\in \{u>0 \}$ and $B_r(x_0)\subset\subset B_1$ we have
\begin{equation}\label{T1}
\sup_{\{u>0 \}\cap\partial B_r(x_0)}u\ge u(x_0)+Cr^2\left(r+|x_0|  \right),
\end{equation}
and
\begin{equation}\label{grow}
u(x)\le C\left( u(x_0)+ r^2(r+|x_0|) \right) \quad\textrm{for }\quad x\in B_{\frac{r}{2}}(x_0).
\end{equation}
\end{lem}

\section{Epiperimetric inequality}		\label{sect Epiperimetric inequality}

Inspired by Weiss's outstanding work \cite{Wei1}, we shall establish the corresponding  epiperimetric inequality for obstacle problem with a degenerate forcing term.
\\

\begin{proof}[Proof of Theorem \ref{TH1}.]
 Assume by contradiction that there are sequences $\kappa_m \to 0$, $\delta_m\to0$, $0\le c_m\in W^{1,2}(B_1)$ that is homogeneous of degree $3$ such that
\begin{equation}\label{ep1}
||c_m-h_m||_{W^{1,2}(B_1)}=\delta_m \quad\textrm{for some}\quad h_m\in \mathbb{H}
\end{equation}
and
\begin{equation}\label{ep2}
M(v)>(1-\kappa_m)M(c_m)+\kappa_mM(h_0)
\end{equation}
for all  $v\in W^{1,2}(B_1)$ with $v=c_m$ on $\partial B_1$,
where $\delta_m:=\inf_{ h_m\in \mathbb{H}}||c_m-h_m||_{W^{1,2}(B_1)}$. By rotation invariance of  the norm, after a possible rotation of coordinate axes, we can assume  $h_m:=h_0$ (see the definition \eqref{ff01}), where
\begin{equation}\label{ep3}
 h_0(x)=\left\{ \frac{|x|^3}{9}-\frac{{x_1}^3-3x_1{x_2}^2 }{9} \right\}\chi_{\{ 0<\theta<\frac{2\pi}{3} \}}.
\end{equation}
Subtracting $M(h_0)$ from both sides of equation \eqref{ep2}, one can see that
\begin{equation}\label{ep4}
(1-\kappa_m)(M(c_m)- M(h_0))<M(v)-M(h_0)
\end{equation}
for all $v\in W^{1,2}(B_1)$ with  $v=c_m$ on $\partial B_1$.
Note the  first variation of the energy $M(v)$ at $h_0$ vanishes. In other words,
\begin{equation}\label{ep5}
(\delta M(h_0))(\phi) = 2\int_{B_1}\left( \nabla h_0 \cdot \nabla\phi+|x|\phi\chi_{\{h_0>0\}}  \right)dx -6\int_{\partial B_1}h_0\phi d\mathcal{H}
\end{equation}
for all $\phi\in W^{1,2}(B_1)$.
Since $h_0$ is a homogeneous function of degree $3$, we have $x\cdot\nabla h_0=3h_0$.  This together the fact $\Delta h_0=|x|\chi_{\{h_0>0\}}$ in $\mathbb{R}^2$  yields
\begin{equation}\label{ep6}
(\delta M(h_0))(\phi)=0 \quad\textrm{ for all}\quad \phi\in W^{1,2}(B_1).
\end{equation}
Hence by taking $\phi=c_m-h_0$, we conclude that
\begin{equation}\label{ep7}
\begin{split}
(1-\kappa_m)(M(c_m)- M(h_0))&=(1-\kappa_m)\left[ M(c_m)- M(h_0)-(\delta M(h_0))(  c_m-h_0 )\right]\\
&=(1-\kappa_m) \bigg[ \int_{B_1} |\nabla(c_m-h_0)|^2dx+ \int_{B_1\cap \{h_0=0 \}} 2|x|c_mdx\\
&\quad-3\int_{\partial B_1}(c_m-h_0)^2d\mathcal{H}\bigg].
\end{split}
\end{equation}
By a similar method, one can see that
\begin{equation}\label{ep8}
\begin{split}
M(v)- M(h_0)&=M(v)- M(h_0)-(\delta M(h_0))( v-h_0 )\\
&=\int_{B_1} |\nabla(v-h_0)|^2dx  +\int_{B_1}2|x| \big[ v^+-h_0-(v-h_0)\chi_{\{ h_0>0  \}} \big]dx             \\
&\quad -3\int_{\partial B_1}(v-h_0)^2d\mathcal{H}\\
&=\int_{B_1} |\nabla(v-h_0)|^2dx-3\int_{\partial B_1}(v-h_0)^2d\mathcal{H}\\
&\quad+\int_{B_1}2|x|\Big[v^+\chi_{\{ h_0=0  \}}+v^-\chi_{\{ h_0>0  \}}\Big]dx.
\end{split}
\end{equation}
Therefore, combining \eqref{ep4}, \eqref{ep7}, and \eqref{ep8} we get
\begin{equation}\label{ep9}
\begin{split}
(1-\kappa_m)& \bigg[ \int_{B_1} |\nabla(c_m-h_0)|^2dx+ \int_{B_1\cap \{h_0=0 \}} 2|x|c_mdx
-3\int_{\partial B_1}(c_m-h_0)^2d\mathcal{H}\bigg]\\
&<  \int_{B_1} |\nabla(v-h_0)|^2dx-3\int_{\partial B_1}(v-h_0)^2d\mathcal{H}+\int_{B_1}2|x|\Big[v^+\chi_{\{ h_0=0  \}}+v^-\chi_{\{ h_0>0  \}}\Big]dx
\end{split}
\end{equation}
for all $v\in W^{1,2}(B_1)$ with  $v=c_m$ on $\partial B_1$.

Define now the sequence of functions $w_m:=\frac{c_m-h_0}{\delta_m}$. Then it is easy to find
\begin{equation}\label{ep10}
||w_m||_{W^{1,2}(B_1)}=1.
\end{equation}
Thus there exists a subsequence $\{w_m\}$ satisfying  $w_m\rightharpoonup w$ weakly in $W^{1,2}(B_1)$.
To reach a contradiction, we need to show the following
\\

\noindent \textbf{Claim.} The sequence of functions $w_m\to w$ strongly in $W^{1,2}(B_1)$, and $w\equiv0$ in $B_1$.
\\

Indeed, suppose the claim holds, then we have
\begin{equation*}
||w_m||_{W^{1,2}(B_1)}=0,
\end{equation*}
which contradicts the \eqref{ep10}  and completes the proof of Theorem \ref{TH1}.

Now, we divide the proof of the claim into four steps.
\\

\noindent \textbf{Step 1.} $w=0$ a.e. in $B_1\cap \{h_0=0 \}^\circ $ and $w=0$ $\mathcal{H}^1$-a.e. on $B_1\cap \partial\{h_0=0 \}$.

Firstly, we show that
\begin{equation}\label{ep11}
\int_{B_1\cap \{h_0=0 \}}\frac{c_m}{{\delta_m}^2}\le C.
\end{equation}
Given a radial function $\eta\in C^{\infty}_0(B_1)$ with $\eta\in[0,1]$, then by inserting $v:=(1-\eta)c_m+\eta h_0$ into \eqref{ep9}, we have
\begin{equation}\label{ep12}
\begin{split}
(1-\kappa_m)\int_{B_1\cap \{h_0=0 \}} 2|x|c_mdx&<3(1-\kappa_m)\int_{\partial B_1}(c_m-h_0)^2d\mathcal{H}+\int_{B_1} |\nabla\Large((1-\eta)(c_m-h_0)\Large)|^2dx\\
&\quad -3\int_{\partial B_1}|(1-\eta) (c_m-h_0)|^2d\mathcal{H}+\int_{B_1\cap \{h_0=0 \}} 2(1-\eta)|x|c_mdx.
\end{split}
\end{equation}
This together with $w_m=\frac{c_m-h_0}{\delta_m} $ yields taht
\begin{equation}\label{ep13}
\begin{split}
(1-\kappa_m)\int_{B_1\cap \{h_0=0 \}} \frac{2|x|}{\delta_m}w_mdx&<3(1-\kappa_m)\int_{\partial B_1}{w_m}^2d\mathcal{H}+\int_{B_1} |\nabla\Large((1-\eta)w_m\Large)|^2dx\\
&\quad +\int_{B_1\cap \{h_0=0 \}} (1-\eta)\frac{2|x|}{\delta_m}w_mdx.
\end{split}
\end{equation}
In view of $||w_m||_{W^{1,2}(B_1)}=1$, one can see that there exists an uniform constant $C_0>0$ such that
\begin{equation}\label{ep14}
3(1-\kappa_m)\int_{\partial B_1}{w_m}^2d\mathcal{H}+\int_{B_1} |\nabla((1-\eta)w_m)|^2dx\le C_0.
\end{equation}
Thus, combining \eqref{ep13} and \eqref{ep14}, we conclude that
\begin{equation}\label{ep15}
\int_{B_1\cap \{h_0=0 \}} 2|x|(\eta-\kappa_m)\frac{w_m}{\delta_m}dx\le C_0.
\end{equation}

Based on the fact $\eta(x)=\eta(|x|)$, and $c_m$ is third-order homogeneous global function, we derive that
\begin{equation}\label{ep16}
\begin{split}
\int_{0}^{1}dt\int_{\partial B_t\cap \{h_0=0 \}}2t(\eta(t)-\kappa_m)\frac{c_m}{{\delta_m}^2}d\mathcal{H}&=\int_{0}^{1}2t^5(\eta(t)-\kappa_m)dt\int_{\partial B_1\cap \{h_0=0 \}}\frac{c_m}{{\delta_m}^2}d\mathcal{H}\\
&=2\int_{0}^{1}t^5(\eta(t)-\kappa_m)dt\int_{\partial B_1\cap \{h_0=0 \}}\frac{w_m}{\delta_m}d\mathcal{H}\\
&\le C_0.
\end{split}
\end{equation}
Note that for $m$ large enough, we have
\begin{equation}\label{ep17}
0<2\int_{0}^{1}t^5(\eta(t)-\kappa_m)dt<C_1
\end{equation}
for some $C_1>0$. Hence one can see that
\begin{equation}\label{ep18}
\int_{\partial B_1\cap \{h_0=0 \}}\frac{w_m}{\delta_m}d\mathcal{H}<C_2.
\end{equation}
In view of $w_m\ge0$ in $\{h_0=0\}$, by taking $m\to\infty$, we deduce that $w(x)=0$ a.e. $\{h_0=0\}^\circ\cap B_1$.\\

\noindent \textbf{Step 2.} $\Delta w=0$ in $B_1\cap \{h_0>0 \} $.\\

Choosing $v=(1-\eta)c_m+\eta(h_0+\delta_m\Psi) $, where $\eta\in C^{\infty}_0(B_1\cap\{h_0>0 \})$ and $\eta(x)\in[0,1]$, $\Psi\in L^{\infty}(B_1)\cap W^{1,2}(B_1)$. It is easy to find that $v\in W^{1,2}(B_1)$ and $v=c_m$ on $\partial B_1$.

Inserting this $v$ into \eqref{ep9}, we have
\begin{equation}\label{ep19}
\begin{split}
(1-\kappa_m)\Bigg[\int_{B_1}|\nabla w_m|^2dx&+\int_{B_1\cap \{h_0=0 \}} \frac{2|x|}{\delta_m} w_m dx \Bigg]\le \int_{B_1}\Bigg| \frac{\nabla( (1-\eta)(c_m-h_0)+\eta\delta_m\Psi )}{\delta_m} \Bigg|^2dx\\
& \quad+\int_{B_1\cap \{h_0>0 \}}\frac{2|x|}{{\delta_m}^2}( (1-\eta)c_m+\eta(h_0+\delta_m\Psi ) )^- dx \\
&\quad+\int_{B_1\cap \{h_0=0 \}}\frac{2|x|}{{\delta_m}^2}((1-\eta)c_m+\eta\delta_m\Psi)^+dx.
\end{split}
\end{equation}
Noticing that $\eta\equiv0$ in $B_1\cap \{h_0=0 \}$, and $\Psi\in L^{\infty}(B_1)$, thus for $m$ large enough, we infer that
\begin{equation}\label{ep20}
\int_{B_1\cap \{h_0=0 \}}\frac{2|x|}{{\delta_m}^2}((1-\eta)c_m+\eta\delta_m\Psi)^+dx=\int_{B_1\cap \{h_0=0 \}}\frac{2|x|}{{\delta_m}^2}c_mdx
\end{equation}
and
\begin{equation}\label{ep21}
\int_{B_1\cap \{h_0>0 \}}\frac{2|x|}{{\delta_m}^2} \left( (1-\eta)c_m+\eta(h_0+\delta_m\Psi ) \right)^- dx=0.
\end{equation}
Combining \eqref{ep19},\eqref{ep20} and \eqref{ep21}, it follows that
\begin{equation}\label{ep22}
\begin{split}
\int_{B_1}|\nabla w_m|^2dx &\le \kappa_m\int_{B_1}|\nabla w_m|^2dx+\kappa_m\int_{B_1\cap \{h_0=0 \}}\frac{2|x|}{\delta_m}w_mdx\\
&\quad+\int_{B_1} |\nabla((1-\eta)w_m+\eta\Psi)|^2dx.
\end{split}
\end{equation}
Letting $m\to \infty$, based on the fact $||w_m||_{W^{1,2}(B_1)}=1$ and with the help of step1, we get
\begin{equation}\label{ep23}
\int_{B_1}|\nabla w|^2dx \le \int_{B_1} |\nabla((1-\eta)w+\eta\Psi)|^2dx.
\end{equation}
Actually, by the classical approximation theory, we can drop the condition $\Psi\in L^\infty(B_1)$.

For all ball $B_0\subset\subset B_1\cap \{h_0>0 \}$, we can choose $\eta\equiv1$ in $B_0$, $\Psi=w$ in $B_1 \verb+\+ B_0$.
Hence we  can see that
\begin{equation}\label{ep24}
\int_{B_0}|\nabla w|^2dx \le \int_{B_0} |\nabla \Psi|^2dx
\end{equation}
for all $\Psi\in W^{1,2}(B_1)$ with $\Psi=w$ on $\partial B_0$.
This together with the arbitrariness of the choice of  $B_0$ yields the desired conclusion.\\

\noindent \textbf{Step 3.} $w=0$ in $B_1$.\\
\indent We already showed that $w=0$ a.e. in $B_1\cap \{h_0=0 \} $. Thus we just need prove $w=0$ in $B_1\cap \{h_0>0 \} $. Using the homogeneity of $w$ and with the help of step 2,we obtain that
$y''(\theta)+9y(\theta)=0$, where $y(\theta)=w(e^{i\theta})$. This leads to $w(e^{i\theta})=A\cos(3\theta)+B\sin(3\theta)$, where $A$ and $B$ are arbitrary constants. Notice that from step 1 we have $w(e^{i\theta})=0$  $\mathcal{H}^1$-a.e. on $B_1\cap \partial\{h_0=0 \}$, and hence one can see that $A=0$ and
\begin{equation}\label{1d03}
w(x)=B(3{x_1}^2x_2-{x_2}^3).
\end{equation}

\indent On the other hand, based on the fact that $h_0$ is the best approximation in $W^{1,2}(B_1)$, we have that
\begin{equation*}
\begin{split}
||c_m-h_0||^2_{W^{1,2}(B_1)}&\le ||c_m-h_0+h_0-h_{\theta_1} ||^2_{W^{1,2}(B_1)}\\
&=||c_m-h_0||^2_{W^{1,2}(B_1)}+||h_{\theta_1}-h_0||^2_{W^{1,2}(B_1)}-2(c_m-h_0, h_{\theta_1}-h_0),
\end{split}
\end{equation*}
witch implies that
\begin{equation}\label{ep25}
(w_m,h_{\theta_1}-h_0)\le \frac{1}{2\delta_m}||h_{\theta_1}-h_0||^2_{W^{1,2}(B_1)}.
\end{equation}
Define $v_{\theta_1}=(\cos\theta_1,\sin\theta_1)$. We denote $v_{\theta_0}$ by $v_0$ for simplicity. Observe that $h_{\theta_1}\to h_0$ if and only  if $\theta_1\to 0+$ or $\theta_1\to 2\pi-$. \\

$(\romannumeral 1)$ In the case: $\theta_1\to 0+$\\

\noindent It is easy to see that
\begin{equation}\label{ep26}
\begin{split}
\frac{(w_m,h_{\theta_1}-h_0)}{|v_{\theta_1}-v_0|}&=\frac{1}{{|v_{\theta_1}-v_0|}}\int_{B_1} \left(w_m(h_{\theta_1}-h_0)+\nabla{w_m}\cdot\nabla(h_{\theta_1}-h_0) \right)dx\\
&=\frac{1}{{|v_{\theta_1}-v_0|}}\int_{B_1\cap D_1}\left(w_m(h_{\theta_1}-h_0)+\nabla{w_m}\cdot\nabla(h_{\theta_1}-h_0) \right)dx\\
&\quad+\frac{1}{{|v_{\theta_1}-v_0|}}\int_{B_1\cap D_2}\left(w_m h_{\theta_1}+\nabla{w_m}\cdot\nabla h_{\theta_1}\right)dx\\
&\quad -\frac{1}{{|v_{\theta_1}-v_0|}}\int_{B_1\cap D_3}\left(w_m h_0+\nabla{w_m}\cdot\nabla h_0\right)dx.
\end{split}
\end{equation}
Here $D_1:=\{x=(r,\theta): r>0,\theta_1<\theta<\frac{2\pi}{3}\}$, $D_2:=\{x=(r,\theta): r>0,\frac{2\pi}{3}<\theta<\theta_1+\frac{2\pi}{3} \}$, $D_3:=\{x=(r,\theta): r>0, 0<\theta<\theta_1 \}$.

For every $x\in D_1$ we have
\begin{equation*}
\begin{split}
\frac{h_{\theta_1}-h_0}{|v_{\theta_1}-v_0|}&=\frac{r^3}{9}\frac{\cos3\theta-\cos3(\theta-\theta_1)}{|v_{\theta_1}-v_0|}\\
&=\frac{r^3}{9}\frac{v_{\theta_1}-v_0}{|v_{\theta_1}-v_0|}\left( (1+2\cos\theta_1)^2 \cos3\theta, (3-4\sin^2\theta_1) \sin3\theta \right).
\end{split}
\end{equation*}
It is easy to see that
\begin{equation*}
\frac{v_{\theta_1}-v_0}{|v_{\theta_1}-v_0|}\to (0,-1) \quad\textrm{as}\quad \theta_1\to 0+.
\end{equation*}
Thus, combining the computations above, we arrive at
\begin{equation}\label{ep27}
\frac{h_{\theta_1}-h_0}{|v_{\theta_1}-v_0|}\to -\frac{r^3}{3}\sin3\theta=-{x_1}^2x_2+\frac{1}{3}{x_2}^3 \quad\textrm{for all}\quad x\in D_1
\end{equation}
as $\theta_1\to 0+$. On the other hand, we have
\begin{equation}\label{ep28}
\begin{split}
\frac{\nabla(h_{\theta_1}-h_0)}{|v_{\theta_1}-v_0|}&=\frac{r^2}{3|v_{\theta_1}-v_0|}
\begin{pmatrix}
(1-\cos3\theta_1)( \cos^2\theta-\sin^2\theta )-2\sin3\theta_1 \sin\theta\cos\theta
\\ (-\sin3\theta_1)( \cos^2\theta-\sin^2\theta )-2(1-\cos3\theta_1)\sin\theta\cos\theta
\end{pmatrix}'\\
& =\frac{r^2}{3}\left[\frac{(1-\cos3\theta_1)(\cos2\theta,-\sin2\theta )}{|v_{\theta_1}-v_0|}-\frac{(\sin3\theta_1)( \sin2\theta, \cos2\theta)}{|v_{\theta_1}-v_0|} \right].
\end{split}
\end{equation}
Note that
\begin{equation}\label{ep29}
\frac{1-\cos3\theta_1}{|v_{\theta_1}-v_0|} =\frac{\sqrt{1-\cos\theta_1}(1+4(1-\cos\theta_1)\cos\theta_1 )}{\sqrt{2}}\to0 \quad\textrm{as}\quad \theta_1\to 0+,
\end{equation}
and
\begin{equation}\label{ep30}
\frac{\sin3\theta_1}{|v_{\theta_1}-v_0|} =(4\cos^2\theta_1 -1 )\cos{\frac{\theta_1}{2}}  \to 3 \quad\textrm{as}\quad \theta_1\to 0+,
\end{equation}
Hence, combining \eqref{ep28}, \eqref{ep29}, and \eqref{ep30}, we obtain
\begin{equation}\label{ep31}
\frac{\nabla(h_{\theta_1}-h_0)}{|v_{\theta_1}-v_0|}\to-r^2(\sin2\theta, \cos2\theta)=(-2x_1x_2,-{x_1}^2+{x_2}^2) \quad\textrm{for all}\quad x\in D_1
\end{equation}
as $\theta_1\to 0+$.

Since $\cos3\theta_1<\cos3(\theta-\theta_1)$ in $D_2$, one can find that
\begin{equation}\label{ep32}
\begin{split}
0&\le\frac{h_{\theta_1}}{|v_{\theta_1}-v_0|}=\frac{r^3}{9}\frac{1-\cos3(\theta-\theta_1)}{\sqrt{2(1-\cos\theta_1 )}}\\
&\le \frac{r^3}{9} \frac{1-\cos3\theta_1}{2\sin{\frac{\theta_1}{2}}}=\frac{r^3}{9} \frac{\sin^2{ \frac{3\theta_1}{2} }}{\sin{\frac{\theta_1}{2}}}\to0   \quad\textrm{as}\quad \theta_1\to 0+.
\end{split}
\end{equation}
Similarly, we have
\begin{equation}\label{ep33}
\begin{split}
\frac{h_0}{|v_{\theta_1}-v_0|}\to0 \quad\textrm{for all}\quad x\in D_3
\end{split}
\end{equation}
as $\theta_1\to 0+$.

On the other hand, by the Cauchy-Schwartz inequality, we deduce that
\begin{equation*}
\begin{split}
\frac{1}{|v_{\theta_1}-v_0|}\left|\int_{B_1\cap D_3}\nabla{w_m}\cdot\nabla h_0dx \right|&\le  ||w_m||_{W^{1,2}(B_1)}\frac{||\nabla h_0||_{L^2(B_1\cap D_3)}}{|v_{\theta_1}-v_0|}.
\end{split}
\end{equation*}
Indeed,
\begin{equation*}
\begin{split}
\frac{||\nabla h_0||_{L^2(B_1\cap D_3)}}{|v_{\theta_1}-v_0|}&=\frac{1}{|v_{\theta_1}-v_0|}\left(\int_{0}^{1}\frac{2r^5}{9}dr\int_{0}^{\theta_1}( 1-\cos3\theta)  d\theta \right)^{\frac{1}{2}}\\
&=\frac{\sqrt{2}}{18}\sqrt{\frac{3\theta_1-\sin3\theta_1}{1-\cos\theta_1}}\to 0 \quad\textrm{as}\quad \theta_1 \to 0+.
\end{split}
\end{equation*}
Hence we derive that
\begin{equation}\label{ep34}
\begin{split}
\frac{1}{|v_{\theta_1}-v_0|}\int_{B_1\cap D_3}\nabla{w_m}\cdot\nabla h_0dx \to 0 \quad\textrm{as}\quad \theta_1 \to 0+.
\end{split}
\end{equation}
Similarly, we have
\begin{equation}\label{ep35}
\begin{split}
\frac{1}{|v_{\theta_1}-v_0|}\int_{B_1\cap D_2}\nabla{w_m}\cdot\nabla h_{\theta_1}dx \to0\quad\textrm{as}\quad \theta_1 \to 0+.
\end{split}
\end{equation}

Actually, the above estimates indicate that there exists a uniform constant $C>0$ such that
\begin{equation*}
\begin{split}
\frac{|h_{\theta_1}-h_0|+ |\nabla(h_{\theta_1}-h_0)|}{|v_{\theta_1}-v_0|}\le C \quad\textrm{for all}\quad x\in B_1
\end{split}
\end{equation*}
as $\theta_1\to 0+$. Furthermore, this leads to the conclusion
\begin{equation}\label{ep36}
\frac{||h_{\theta_1}-h_0||^2_{W^{1,2}(B_1)}}{|v_{\theta_1}-v_0|}\to0 \quad\textrm{as}\quad \theta_1 \to 0+.
\end{equation}
Multiplying both sides of \eqref{ep25} by $\frac{1}{|v_{\theta_1}-v_0|}$, and with the aid of \eqref{ep27}, \eqref{ep31}-\eqref{ep36}, one can deduce that
\begin{equation}\label{ep37}
\begin{split}
\int_{B_1} w_m (-{x_1}^2x_2+\frac{{x_2}^3}{3})\chi_{\{ h_0>0 \}}dx+\int_{B_1}\nabla{w_m}\cdot(-2x_1x_2,-{x_1}^2+{x_2}^2)\chi_{\{ h_0>0 \}}dx \le o(1)
\end{split}
\end{equation}
as $\theta_1\to 0+$.  Further, using \eqref{1d03}, and then taking $m\to \infty$, we obtain           \begin{equation*}
\begin{split}
-B\int_{B_1\cap \{h_0>0\}}\left(\frac{(3{x_1}^2x_2-{x_2}^3)^2}{3}+3|x|^4 \right)dx  \le 0.
\end{split}
\end{equation*}
This implies that
\begin{equation}\label{ep38}
B\ge0.
\end{equation}

$(\romannumeral 2)$ In the case: $\theta_1\to 2\pi-$\\

\noindent By a similar method in case $(\romannumeral 1)$, firstly letting $\theta_1\to 2\pi-$,  we have
\begin{equation*}
\begin{split}
\int_{B_1} w_m ({x_1}^2x_2-\frac{{x_2}^3}{3})\chi_{\{ h_0>0 \}}dx+\int_{B_1}\nabla{w_m}\cdot(2x_1x_2,{x_1}^2-{x_2}^2)\chi_{\{ h_0>0 \}}dx \le o(1)
\end{split}
\end{equation*}
And then taking $m\to \infty $, one can derive that
\begin{equation*}
\begin{split}
B\int_{B_1}\left( \frac{(3{x_1}^2x_2-{x_2}^3)^2}{3}+3|x|^4\right)\chi_{\{ h_0>0 \}} dx \le 0,
\end{split}
\end{equation*}
which implies
\begin{equation*}
B\le0.
\end{equation*}
This together with \eqref{ep38} yields to $B=0$, which proves the desired conclusion.\\

Next, we claim that the  ${w_m}$ goes to $0$ as the subsequence $m\to \infty$, which is sufficient to arrive at a contradiction to the definition ${w_m}$. Note that we still denote this subsequence as ${w_m}$.\\

\noindent \textbf{Step 4.} $w_{m}\to w=0$ strongly  in $W^{1,2}(B_1)$ as $m\to \infty$.

Since $w_{m}\to w $ weakly in $W^{1,2}(B_1)$ and  $W^{1,2}(B_1) \subset\subset L^2(B_1)$, one can find the subsequence, still denoted by  $w_{m}$, such that $w_{m}\to w$ strongly  in $L^2(B_1)$ as $m\to \infty$. Therefore, we still need to show the strong convergence of $\nabla w_{m}$ in $L^2(B_1)$ as the subsequence $m\to \infty$.

To this end, one can choose a test function $v:=(1-\eta)c_m+\eta(x) h_0$ with $\eta(x)=\min (\max(2-2|x|,0),1)$.
It is easy to see that $\eta\in W^{1,\infty}_0(B_1)$ and $v=c_m$ on $\partial B_1$. Now, plugging $\eta(x)$ into \eqref{ep9}, we obtain
\begin{equation*}
\begin{split}
(1-\kappa_m)\int_{B_1} |\nabla w_m|^2dx&<\int_{B_1\cap \{h_0=0 \}}\frac{2(\kappa_m-\eta)|x|w_m}{\delta_m}dx+3\int_{\partial B_1}(1-\kappa_m-(1-\eta)^2){w_m}^2d\mathcal{H}\\
&\quad+\int_{B_1}|\nabla((1-\eta)w_m)|^2dx.
\end{split}
\end{equation*}
Furthermore, since $\eta=0$ on $\partial B_1$, we conclude that
\begin{equation}\label{ep39}
\begin{split}
\int_{B_1} |\nabla w_m|^2dx&\le C\kappa_m+\int_{B_1} |\nabla \eta|^2{w_m}^2dx-2\int_{B_1}(1-\eta)w_m\nabla\eta\cdot\nabla w_m dx\\
&\quad +\int_{B_1}(1-\eta)^2|\nabla w_m|^2dx.
\end{split}
\end{equation}
for large $m$. Thanks to the choice of $\eta(x)$, one can see that
\begin{equation}\label{ep40}
\int_{B_1} |\nabla \eta|^2{w_m}^2dx\le C||w_m||^2_{L^2(B_1)},
\end{equation}
and
\begin{equation}\label{ep41}
\int_{B_1}(1-\eta)^2|\nabla w_m|^2dx\le \int_{B_1 \backslash B_{1/2}}|\nabla w_m|^2dx.
\end{equation}
Using the Cauchy-Schwartz inequality and the fact that $||w_m||_{W^{1,2}(B_1)}=1$, we obtain that
\begin{equation}\label{ep42}
\begin{split}
-2\int_{B_1}(1-\eta)w_m\nabla\eta\cdot\nabla w_m dx&\le C ||\nabla w_m||_{L^2(B_1)}||w_m||_{L^2(B_1)}\\
&\le C ||w_m||_{L^2(B_1)}.
\end{split}
\end{equation}
And therefore, collecting  \eqref{ep39}-\eqref{ep42}, we arrive at
\begin{equation}\label{ep43}
\begin{split}
\int_{B_{1/2}}|\nabla w_m|^2dx\le C\left(\kappa_m+||w_m||^2_{L^2(B_1)}\right).
\end{split}
\end{equation}
Moreover, applying the homogeneity of $w_m$ again, it follows from \eqref{ep43} that
\begin{equation}\label{ep44}
\begin{split}
\int_{B_{1/2}}|\nabla w_m|^2dx&= \int_{0}^{1/2}{\rho}^5d\rho\int_{\partial B_1}|\nabla w_m|^2d\mathcal{H}\\
&=\frac{1}{384}\int_{\partial B_1}|\nabla w_m|^2d\mathcal{H}\\
&=\frac{1}{64}\int_{B_1}|\nabla w_m|^2dx.
\end{split}
\end{equation}
Now, putting \eqref{ep43} and \eqref{ep44} together, we prove that
\begin{equation*}
\int_{B_1}|\nabla w_m|^2dx\le C\left(\kappa_m+||w_m||^2_{L^2(B_1)}\right) \to 0 \quad\textrm{as}\quad m \to \infty.
\end{equation*}
This finishes the proof of the Claim, and hence completes the proof of Theorem \ref{TH1}.
\end{proof}

\section{$C^{1,1}$ bound and the  monotonicity formula} \label{monotonicity formula}

Now, we are in  a position  to show the monotonicity formula. Firstly, based on the
optimal growth estimate \eqref{grow}, one can show the growth behavior of the solution away from the free boundary.
\begin{lem}\label{L4.1}
Let $u$ be a solution of \eqref{f6} and suppose that $x_0\in \{u>0\}$, $B_{5d}(x_0)\subset B_1$, then we have
\begin{equation*}
u\le Cd^2(|x_0|+d)  \quad\textrm{in}\quad B_d(x_0).
\end{equation*}
where $d:=dist(x_0,\partial \{u=0\})$.
\end{lem}
\begin{proof}
Let $y\in \partial \{u=0\}$ satisfying $d=dist(x_0,\partial \{u=0\})=|y-x_0|$. It is easy to see that $B_{4d}(y)\subset B_{5d}(x_0)$, hence one can derive form \eqref{grow} that
\begin{equation*}
u\le Cd^2(|y|+d) \quad\textrm{in}\quad B_{2d}(y)
\end{equation*}
thanks to $u(y)=0$.  Since $B_d(x_0)\subset B_{2d}(y)$, by the triangle inequality, it follows that
\begin{equation*}
\begin{split}
u&\le Cd^2(d+|y-x_0|+|x_0|)\\
&=Cd^2(2d+|x_0|) \quad\textrm{for all}\quad  x\in B_d(x_0).
\end{split}
\end{equation*}
This completes the proof of Lemma \ref{L4.1}.
\end{proof}

Next, we will show the sequence $\{ u_r \}$ is uniformly bounded in $C^{1,1}(B_1)$. In our problem, this property is not obvious due to the definition of $u_r$.  With the uniform boundedness in hand, one can consider a possible limit as $r\to 0+$.
\begin{lem}\label{L4.2}
Let $u$ be a solution of \eqref{f6}. Suppose  $0$ is a regular point and $B_{r_0}\subset B_1$,  then there exists $C>0$ such that
\begin{equation}\label{m1}
||u_r||_{C^{1,1}(B_1)} \le C  \quad\textrm{for all}\quad r\in (0,r_0/6).
\end{equation}
\end{lem}
\begin{proof}
 Fix $x\in \{ u>0 \}\cap B_{\frac{r_0}{6}}$. And define $d:=dist(x,\partial\{ u=0 \} )$. It's obvious that $d\le |x|$ and $B_{5d}(x)\subset B_{5d+|x|}(0)\subset B_{6|x|}(0)\subset B_{r_0}(0)$. Thus, by the optimal growth estimate \eqref{grow}, one have
\begin{equation*}
u(y)\le Cd^2(|x|+d  ) \quad\textrm{for all}\quad y\in B_d(x),
\end{equation*}
which implies that
\begin{equation}\label{m2}
||u||_{L^\infty (B_d(x))}\le Cd^2(|x|+d).
\end{equation}
Now, for every $x\in B_r\cap \{u>0\}$ with $0<r<\frac{r_0}{6}$, thanks to \eqref{m2} and $d\le |x| $ we have
\begin{equation*}
\begin{split}
|u(x)|&\le ||u||_{L^\infty B_d(x)} \\
&\le Cd^2(|x|+d)\\
&\le  C|x|^3\\
&< Cr^3.
\end{split}
\end{equation*}
This says that
\begin{equation}\label{m3}
||u||_{L^\infty (B_r)}\le Cr^3.
\end{equation}
In order to establish gradient estimates, we firstly define
\begin{equation*}
h(x):=\frac{1}{9}|x|^3, \quad x\in \mathbb{R}^2,
\end{equation*}
and
\begin{equation*}
P_{x_0}(x):=h(x)-h(x_0)-\nabla h(x_0)(x-x_0), \quad \textrm{for all}\quad x,\ x_0\in \mathbb{R}^2.
\end{equation*}
Let $V=(u-P_x)$, then for $x\in B_r\cap \{u>0\}$, we have
\begin{equation*}
\Delta V=0 \quad \textrm{in} B_d(x).
\end{equation*}
Therefore, by Schauder estimates we arrive at
\begin{equation}\label{m4}
\begin{split}
 |\nabla u(x)|=|\nabla V(x)|&\le\frac{C}{d}||V||_{L^\infty (B_d(x))}\\
 &\le \frac{C}{d}||u||_{L^\infty B_d(x)}+\frac{C}{d}||P_x||_{L^\infty (B_d(x))}.
\end{split}
\end{equation}
On the other hand, by a simple calculation we have
\begin{equation*}
\begin{split}
P_{x}(y)= \frac{1}{9} \left(|y|^3-|x|^3-3|x|x\cdot(y-x) \right) \quad\textrm{for all}\quad y\in \mathbb{R}^2.
\end{split}
\end{equation*}
Furthermore, we have
\begin{equation*}
\begin{split}
\left||y|^3-|x|^3-3|x|x\cdot(y-x) \right|=6\left| \int_{0}^{1}(1-t)(|x+t(y-x)||y-x|^2)dt \right|
\end{split}
\end{equation*}
Since $|x+t(y-x)|\le |x|+|y-x|$,  one can see that
\begin{equation}\label{m5}
\begin{split}
|P_{x}(y)|&\le \frac{2}{3}(|x|+|y-x|)|y-x|^2 \int_{0}^{1}(1-t)dt\\
&\le |y-x|^2(|x|+|y-x|).
\end{split}
\end{equation}
Putting \eqref{m2}, \eqref{m4} and \eqref{m5} together, we obtain
\begin{equation}\label{mm6}
\begin{split}
|\nabla u(x)|&\le Cd(|x|+d  )\\
&\le 2Cr^2.
\end{split}
\end{equation}
which implies
\begin{equation}\label{m6}
||\nabla u||_{L^\infty (B_r)}\le Cr^2.
\end{equation}

Finally, it remains to show that $[\nabla u]_{C^{0,1}(B_r)}\le Cr$. To this end, we fix $x,y\in B_r$
and split the proof into two case.

$(\romannumeral 1)$ $B_{|x-y|}(z)\subset \{u>0\}$ with $z=\frac{1}{2}(x+y)$.\\
In this case, we have $t:=dist(z,\partial \{u=0\})\ge |x-y|$. Meanwhile, we have
\begin{equation}\label{m7}
\begin{split}
\frac{|\nabla u(x)-\nabla u(y)|}{|x-y|}&\le [\nabla u]_{C^{0,1}(B_{|x-y|/2}(z))}\le  [\nabla u]_{C^{0,1}(B_{t/2}(z))}\\
&\le [\nabla (u -P_z)]_{C^{0,1}(B_{t/2}(z))} +[\nabla P_z]_{C^{0,1}(B_{t/2}(z))}.
\end{split}
\end{equation}
As similar to \eqref{m4}, and with the help of \eqref{m2} and \eqref{m5}, we obtain
\begin{equation}\label{m8}
\begin{split}
[\nabla (u -P_z)]_{C^{0,1}(B_{t/2}(z))}&\le \frac{C}{t^2} ||u -P_z||_{L^\infty (B_{t/2}(z))}\\
&\le  C(t+|z| )\\
&\le 2Cr.
\end{split}
\end{equation}
On the other hand, by a simple calculation, we have
\begin{equation*}
|D^2P_{x_0}(x)|_M= |D^2h(x)|_M=\frac{1}{3}\left| |x|I+\frac{x_1\otimes x_2  }{|x|}  \right|_M=\frac{\sqrt{5}}{3}|x| \quad \textrm{for all}\quad x,\ x_0\in \mathbb{R}^2.
\end{equation*}
Here $|\cdot|_M$ denotes the matrix norm. Therefore, one can derive that
\begin{equation*}
|D^2P_{z}(\cdot)|_M\le \frac{\sqrt{5}}{3}(t+|z| )\le \frac{2\sqrt{5}}{3}r \quad \textrm{in}\quad B_{t/2}(z),
\end{equation*}
Hence we arrive at
\begin{equation*}
||D^2P_z ||_{L^\infty (B_{t/2}(z))}\le \frac{2\sqrt{5}}{3}r,
\end{equation*}
which implies that
\begin{equation}\label{m9}
\begin{split}
[\nabla P_z]_{C^{0,1}(B_{t/2}(z))}\le Cr
\end{split}
\end{equation}
Combining \eqref{m7},\eqref{m8} and \eqref{m9}, we have
\begin{equation}\label{m10}
\frac{|\nabla u(x)-\nabla u(y)|}{|x-y|}\le Cr.
\end{equation}

$(\romannumeral 2)$ $B_{|x-y|}(z)\cap \{u=0\}\neq\emptyset$.\\
Applying \eqref{mm6}, we have
\begin{equation*}
\begin{split}
\frac{|\nabla u(x)-\nabla u(y)|}{|x-y|}&\le \frac{|\nabla u(x)|+|\nabla u(y)|}{|x-y|}\\
&\le \frac{C\left(d_x(|x|+d_x)+d_y(|y|+d_y)\right)}{|x-y|}
\end{split}
\end{equation*}
Here $d_x=dist(x,\partial\{u=0\}), d_y=dist(y,\partial\{u=0\})$. Note that in this case, we have
$d_x\le 2|x-y|$ and $d_y\le 2|x-y|$. Therefore, we can derive that
\begin{equation}\label{m11}
\begin{split}
\frac{|\nabla u(x)-\nabla u(y)|}{|x-y|}&\le C(|x|+|y| )\le Cr.
\end{split}
\end{equation}

Since we can do this for each $x,y\in B_r$ with $x\neq y$, it follows that
\begin{equation}\label{m12}
[\nabla u]_{C^{0,1}(B_r)}\le Cr.
\end{equation}
Note that $||u_r||_{L^\infty (B_1)}=\frac{1}{r^3}||u||_{L^\infty (B_r)}$, $||\nabla u_r||_{L^\infty (B_1)}=\frac{1}{r^2}||\nabla u||_{L^\infty (B_r)}$
and $[\nabla u_r]_{C^{0,1}(B_1)}=\frac{1}{r}[\nabla u]_{C^{0,1}(B_r)}$, we obtain from \eqref{m3}, \eqref{m6} and \eqref{m12}  the desired conclusion.
\end{proof}

Next, we give some basic properties regarding to the monotonicity formula \eqref{f7}.
\begin{lem}\label{L4.3}
Let $u$ be any solution to \eqref{f6}. Then we have
\begin{equation}\label{m13}
\frac{d}{dr} \Phi(r,u)=\frac{2}{r^5}\int_{\partial B_1}\left| \nabla u(rx)\cdot x -\frac{3u(rx)}{r}  \right|^2d\mathcal{H}\ge0.
\end{equation}
For $u$ a third-order homogeneous solution of \eqref{f6}, we  have
\begin{equation}\label{m14}
\Phi(1,u)=\int_{B_1}|x|u^+dx.
\end{equation}
\end{lem}
\begin{proof}
 Firstly, integrating by parts, it is easy to check that
\begin{equation}\label{m15}
\begin{split}
\frac{d}{dr} \Phi(r,u)&=\frac{d}{dr}\Phi(1,u_r)=2\int_{B_1}\nabla u_r\cdot\nabla\left( \frac{\nabla u(rx)\cdot x}{r^3}-\frac{3u(rx)}{r^4}  \right)\\
&+2|x|\chi_{ \{u_r>0\}}\left( \frac{\nabla u(rx)\cdot x}{r^3}-\frac{3u(rx)}{r^4}  \right)dx\\
&\quad-3\int_{\partial B_1}2u_r \left( \frac{\nabla u(rx)\cdot x}{r^3}-\frac{3u(rx)}{r^4}  \right)d\mathcal{H}\\
&=2\int_{B_1}\left(|x|\chi_{ \{u_r>0\}}-\Delta u_r\right)\left( \frac{\nabla u(rx)\cdot x}{r^3}-\frac{3u(rx)}{r^4}  \right)dx\\
&\quad+2\int_{\partial B_1}(\nabla u_r\cdot x-3u_r )\left( \frac{\nabla u(rx)\cdot x}{r^3}-\frac{3u(rx)}{r^4}  \right)d\mathcal{H}\\
&=2\int_{\partial B_1}(\nabla u_r\cdot x-3u_r )\left( \frac{\nabla u(rx)\cdot x}{r^3}-\frac{3u(rx)}{r^4}  \right)d\mathcal{H}\\
&=\frac{2}{r^5}\int_{\partial B_1}\left| \nabla u(rx)\cdot x -\frac{3u(rx)}{r}  \right|^2d\mathcal{H}.
\end{split}
\end{equation}
Finally, as similar to the previous proof, we have
\begin{equation*}
\begin{split}
\Phi(1,u)&=\int_{B_1}\left(|\nabla u|^2+2|x| u^+\right)dx-3\int_{\partial B_1} u^2d\mathcal{H}\\
&=\int_{ B_1}(2|x|u\chi_{\{u>0\}}-u\Delta u)dx+\int_{\partial B_1}(u\partial_{\nu}u-3u^2)d\mathcal{H}\\
&=\int_{ B_1}|x|u\chi_{\{u>0\}}dx
\end{split}
\end{equation*}
provided $u$ is a third-order homogeneous solution.
\end{proof}

Now, we give a direct corollary of Lemma \ref{L4.3}. Since the proof is fairly standard, we will omit it here (see \cite{Wei1}).
\begin{cor}\label{Coro4.4}
Let $u$ be any solution to \eqref{f6}. Then, any blow-up of $u$ at $0$ is homogeneous of degree $3$.
\end{cor}

\section{Classification of global solutions and uniqueness of blowups} \label{Classification and uniqueness}

Next, we want to classify all possible blow-up profiles for the solutions to \eqref{f6}. To this end,
we want to find all solutions of following system
\begin{equation}\label{g1}
\left\{
\begin{split}
&\Delta u=|x|\chi_{\{u>0\}} \textup{ in $\mathbb{R}^2$},\\
&u\,\,\, \textup{third-order homogeneous},\\
&u\ge0,\\
&u\in C^{1,1}_{loc}(\mathbb{R}^2).
\end{split}
\right.
\end{equation}

Let $u$ be any solution to \eqref{f6}, by  Lemma \ref{L4.2} one can derive that there is a subsequence $r_k\to0$
such that $u_{r_k}\to u_0$ in $C^1_{loc}(\mathbb{R}^2)$. Here $u_0$ is a blow-up limit of $u$ at $0$.
Invoke all above lemmas, we have $u_0$ solves \eqref{g1} in the sense of distributions.

On the other hand,  the solutions of \eqref{g1} are the nonnegative solutions of the following  no-sign obstacle problem (see \cite{Pet})
\begin{equation}\label{g2}
\left\{
\begin{split}
&\Delta u=|x|\chi_{\{\Omega(u)\}} \textup{ in $\mathbb{R}^2$},\\
&u\,\,\, \textup{third-order homogeneous},\\
&\Omega(u)= \{u=|\nabla u|=0  \}^c,\\
&u\in C^{1,1}_{loc}(\mathbb{R}^2).
\end{split}
\right.
\end{equation}
Therefore, our goal is to find all the solutions of \eqref{g2}, and then eliminate those that may become negative at some point.
In fact, we don't care the trivial case $u\equiv0$ in \eqref{g2}. Hence by the homogeneity of $u$, we obtain that $\Omega(u)$
can only be an open cone in $\mathbb{R}^2\verb+\+ \{0 \}$, i.e., $\Omega(u)=\mathbb{R}^2\verb+\+ \{0 \}:=\Omega_1(u)$  or $\Omega(u)=\{x=re^{i\theta}| \theta_1<\theta<\theta_2\le\theta_1+2\pi, r>0\}:=\Omega_2(u)$
or the union of disjoint open cones. For simplicity, we sometimes denote $\Omega_2(u)$ as  $\left\{\theta:\theta_1<\theta<\theta_2 \right\}$
Note that in the latter two cases, we have that $u=|\nabla u|=0$ holds just on the boundary of the cone.
\\

Now, we first consider the second scenario, where $\Omega(u)$ consists of a single open cone $\Omega_2(u)$.
\begin{lem}\label{L5.1}
Let $u$ be any solution to \eqref{g2} with an open cone $\Omega(u)=\Omega_2(u)$ if and only if $\theta_1\in [0,2\pi)$ and $\theta_2=\theta_1+\frac{2\pi}{3}$ such that
\begin{equation}\label{solu1}
u=\frac{r^3}{9}(1-\cos3(\theta-\theta_1)) \quad\textrm{in}\quad \Omega(u).
\end{equation}
\end{lem}
\begin{proof}
Define $V(x):=u(x)-\frac{1}{9}|x|^3$, we have $\Delta V=0$  in  $\Omega(u)$. By the homogeneity of $u$,
we have  $V(x)=r^3(A\cos3\theta+B\sin3\theta )$, where $A, B$ are arbitrary constants. Furthermore, we have $u(x)=u(re^{i\theta})=r^3(A\cos3\theta+B\sin3\theta )+\frac{1}{9}r^3$ in $\Omega(u)$.
Hence $u(x)|_{\partial B_1}=u(e^{i\theta})=A\cos3\theta+B\sin3\theta +\frac{1}{9}$, and $\partial_{\theta}u(e^{i\theta})=-3A\sin3\theta+3B\cos3\theta$. For simplicity,
let us define $\varpi_{u}(\theta):=u(e^{i\theta})-\frac{i}{3}\partial_{\theta}u(e^{i\theta})$.
It's easy to check that
\begin{equation}\label{g3}
\varpi_{u}(\theta)=(A-Bi)e^{i3\theta}+\frac{1}{9}.
\end{equation}
As we noted earlier, $u=|\nabla u|=0$ holds just on the boundary $\{ re^{i\theta_1}\}$ and $\{ re^{i\theta_2}\}$, this is equivalent to
$\varpi_{u}(\theta_1)=\varpi_{u}(\theta_2)=\mathbf{0}\in \mathbb{C}$, and $\varpi_{u}(\theta)\neq\mathbf{0}$ for $\theta_1<\theta<\theta_2\le\theta_1+2\pi$.
Firstly, inserting $\varpi_{u}(\theta_1)=\mathbf{0}$ in \eqref{g3}, we have
\begin{equation}\label{g3b}
A-Bi=-\frac{1}{9}e^{-3\theta_1 i}
\end{equation}
Hence we have
\begin{equation}\label{g4}
\varpi_{u}(\theta)=\frac{1}{9}\left(1-e^{3(\theta-\theta_1)i}  \right).
\end{equation}
Similarly, inserting  $\varpi_{u}(\theta_2)=\mathbf{0}$ in \eqref{g4}, one can derive that
\begin{equation}\label{g5}
\theta_2=\theta_1+\frac{2k\pi}{3} \quad\textrm{for all}\quad k\in \mathbb{Z}.
\end{equation}
Note that $u(e^{i\theta})$ is a $2\pi$-periodic function, hence one can restrict $\theta_1\in[0,2\pi)$.
On the other hand, $\theta_2\le \theta_1+2\pi$, hence one have $k=1,2,3$. Furthermore, we can exclude the cases where $k=2$ and $k=3$.
Otherwise we can always find the point $\underline{\theta}=\theta_1+\frac{2\pi}{3}\in (\theta_1,\theta_2)$ such that
$\varpi_{u}(\underline{\theta})=0$, which contradicts the definition of $\Omega_2(u)$.
With the aid of \eqref{g3b}, it is easy to see that
\begin{equation}\label{g6}
u(e^{i\theta})=\frac{1}{9}(1-\cos3(\theta-\theta_1)).
\end{equation}
This completes the proof.
\end{proof}

With Lemma \ref{L5.1} in hand, we can obtain all solutions of \eqref{g1}.
\begin{cor}\label{Coro5.1}
Let $u$ be any solution to \eqref{g1} with an open cone $\Omega(u)=\Omega_2(u)$ if and only if $\theta_1\in [0,2\pi)$ and $\theta_2=\theta_1+\frac{2\pi}{3}$ such that
\begin{equation}\label{solu2}
u=\frac{r^3}{9}(1-\cos3(\theta-\theta_1))  \quad\textrm{for all}\quad \theta_1<\theta<\theta_2.
\end{equation}
\end{cor}
\begin{proof}
It follows that $u\ge0$ from \eqref{g6}. In this case, we have $\Omega(u)= \{u>0\}$. Hence the any solution of \eqref{g2} with an open cone $\Omega(u)=\Omega_2(u)$
solves \eqref{g1}.
\end{proof}

Next, we will consider the setting where the open cone consisting of a union of disjoint
connected open cones.

\begin{lem}\label{L5.2} Let $u$ be any solution to \eqref{g1} with the associated open cone
 consisting of a union of disjoint connected open cones. Then either of the following holds
\begin{flalign}\label{L11001}
&\quad(\romannumeral 1)\quad u=u_1+u_2, &&&
\end{flalign}
and
\begin{equation}\label{L001}
u_1=\frac{r^3}{9}(1-\cos3(\theta-\theta_1)) \quad\textrm{for all}\quad \theta_1<\theta<\theta_2,
\end{equation}
\begin{equation}\label{L002}
u_2=\frac{r^3}{9}(1-\cos3(\theta-(\theta_2+\sigma)))  \quad\textrm{for all}\quad \theta_2+\sigma<\theta<\theta_2+\sigma+\frac{2\pi}{3},
\end{equation}
where $0\le\theta_1<2\pi$,  $0<\sigma\le\frac{2\pi}{3}$.

\begin{flalign}\label{L11002}
&\quad (\romannumeral 2) \quad u=\sum\limits_{i=1}^{3}u_i, &&&
\end{flalign}
and
\begin{equation}\label{L003}
u_i=\frac{r^3}{9}(1-\cos3(\theta-\theta_i)) \quad\textrm{for all}\quad \theta_i<\theta<\theta_i+\frac{2\pi}{3},
\end{equation}
where $\theta_i=\theta_1+\frac{2(i-1)\pi}{3}$, $0\le\theta_1<2\pi$.
\end{lem}
\begin{proof} From the proof of Lemma \ref{L5.1}, it can be observed that when $\Omega(u)\neq\mathbb{R}^2\verb+\+ \{0 \}$,
the only  permitted form of an open cone is exclusively $\{\theta: \theta_1<\theta<\theta_2, \theta_2=\theta_1+\frac{2\pi}{3}, 0\le \theta_1<2\pi  \}$.
As a result, the desired open cone is composed of two or three disjoint open cones of this form.
\end{proof}

Next, we consider the solution of \eqref{g1} with the open cone $\Omega(u)=\Omega_1(u)$.
\begin{lem}\label{L5.3} Let $u$ be any solution to \eqref{g1} with an open cone $\Omega_1(u)=\mathbb{R}^2\verb+\+ \{0 \}$. Then we have
\begin{equation}\label{g7}
u=r^3(A\cos3\theta+B\sin3\theta +\frac{1}{9}) \quad\textrm{for all}\quad 0\le\theta<2\pi,
\end{equation}
where $A, B\in \mathbb{R}$ with $\sqrt{A^2+B^2}<\frac{1}{9}$.
\end{lem}
\begin{proof}
As similar to the proof of Lemma \ref{L5.1}, the solutions $u$  of \eqref{g2} solve \eqref{g1} in $\mathbb{R}^2\verb+\+ \{0 \}$ if and only if

\begin{equation}\label{g8}
\left\{
\begin{split}
&u(e^{i\theta})=A\cos3\theta+B\sin3\theta +\frac{1}{9}\ge0 \quad\textrm{for all}\quad\theta \in \mathbb{R},\\
&\varpi_{u}(\theta)=(A-Bi)e^{i3\theta}+\frac{1}{9}\neq\mathbf{0} \quad\textrm{for all}\quad \theta \in \mathbb{R},
\end{split}
\right.
\end{equation}
where $A,B$ are arbitrary constants. This is equivalent to the statement that any admissible parameter pair $(A, B)$ of \eqref{g8} satisfies either
\begin{equation}\label{g9}
\left\{
\begin{split}
&u(e^{i\theta})=A\cos3\theta+B\sin3\theta +\frac{1}{9}\ge0 \quad\textrm{for all}\quad\theta \in \mathbb{R},\\
&A\cos3\theta+B\sin3\theta +\frac{1}{9}\neq0 \quad\textrm{for all}\quad \theta \in \mathbb{R}.
\end{split}
\right.
\end{equation}
or
\begin{equation}\label{g10}
\left\{
\begin{split}
&u(e^{i\theta})=A\cos3\theta+B\sin3\theta +\frac{1}{9}\ge0 \quad\textrm{for all}\quad\theta \in \mathbb{R},\\
&A\sin3\theta-B\cos3\theta \neq0 \quad\textrm{for all}\quad \theta \in \mathbb{R}.
\end{split}
\right.
\end{equation}
It is easy to verify that the admissible parameter pair $(A,B)$ of \eqref{g9} meets $\sqrt{A^2+B^2}< \frac{1}{9}$.
And the admissible parameter pair $(A,B)$ of \eqref{g10} is empty. Hence all admissible parameter pairs $(A,B)$ of \eqref{g8}
satisfy $\sqrt{A^2+B^2}< \frac{1}{9}$. Note that $u(e^{i\theta})$ is a $2\pi$-periodic function, hence one can restrict $\theta\in[0,2\pi)$. Finally, the  homogeneity gives the desired conclusion.
\end{proof}

Thanks to Lemma \ref{L4.3}, one can give the Weiss energy for each solution.
\begin{lem}\label{L5.4} Let $u$ be any solution to \eqref{g1}.

$(\romannumeral 1)$ If $u$ meets \eqref{solu2}, then we have $\Phi(1,u)=\frac{\pi}{81}$.

$(\romannumeral 2)$ If $u$ meets \eqref{L11001}, then we have $\Phi(1,u)=\frac{2\pi}{81}$.

$(\romannumeral 3)$ If $u$ meets \eqref{L11002} or \eqref{g7}, then we have $\Phi(1,u)=\frac{\pi}{27}$.
\end{lem}
\begin{proof}
With the help of  \eqref{m14}, we have
\begin{equation*}
\Phi(1,u)=\int_{0}^{1}r^2dr\int_{0}^{2\pi}u\chi_{\{u> 0\}}d\theta.
\end{equation*}
Hence, If $u$ meets \eqref{solu2}, then we have
\begin{equation}\label{g11}
\Phi(1,u)=\frac{1}{9}\int_{0}^{1}r^5dr\int_{\theta_1}^{\theta_1+\frac{2\pi}{3}}(1-\cos3(\theta-\theta_1))d\theta=\frac{\pi}{81}.
\end{equation}
Similarly, if $u$ meets \eqref{L11001}, then we have $\Phi(1,u)=\frac{2\pi}{81}$.  If $u$ meets \eqref{L11002} or \eqref{g7}, then we have $\Phi(1,u)=\frac{\pi}{27}$.
\end{proof}

\begin{proof}[Proof of Theorem \ref{TH1.2}.]
A combination of Corollary \ref{Coro5.1}, Lemma \ref{L5.2} and Lemma \ref{L5.3} yields the desired conclusion.
\end{proof}

Now, we are in the position to show the uniqueness of the blowup limits at regular free boundary point, which is the very important to establish the regularity of the free boundary.\\

\begin{proof}[Proof of Theorem \ref{TH1.2un}.]
Define the energy error as
\begin{equation*}
e(r):=\Phi(r)-\Phi(0+).
\end{equation*}
In order to establish the decay rate of the energy, it is sufficient to construct the ODE of the form  $e'(r)\ge \frac{c}{r}e(r)$. To proceed further, we differentiate $e(r)$ to obtain
\begin{equation}\label{u1}
e'(r)=-6r^{-7}A(r)+r^{-6}A'(r)+21r^{-8}B(r)-3r^{-7}B'(r),
\end{equation}
where
\begin{equation*}
A(r):=\int_{B_r}\left( |\nabla u|^2+2|x|u^+ \right)dx,
\end{equation*}
and
\begin{equation*}
B(r):=\int_{\partial B_r}u^2d\mathcal{H}.
\end{equation*}
For simplicity, we define $F(x):=|\nabla u|^2 +2|x|u^+  $. Direct computations show
that
\begin{equation}\label{u2}
\begin{split}
A(r)&=r^2\int_{B_1}\left(|\nabla u(ry)|^2+2|ry|u^+(ry)  \right)dy\\
&=r^2 \int_{B_1}F(ry)dy.
\end{split}
\end{equation}
Then, differentiating the identity \eqref{u2} and applying integration by parts give
\begin{equation}\label{u3}
\begin{split}
A'(r)&=2r\int_{B_1}F(ry)dy+r\left(\int_{\partial B_1}F(ry)d\mathcal{H}-2\int_{B_1}F(ry)dy\right)\\
&=r\int_{\partial B_1}F(ry)d\mathcal{H}.
\end{split}
\end{equation}
Similarly, one can deduce that
\begin{equation}\label{u4}
B(r)=r\int_{\partial B_1}u^2(ry)d\mathcal{H},
\end{equation}
and therefore, we have
\begin{equation*}
B'(r)=\int_{\partial B_1}u^2(ry)d\mathcal{H}+2r\int_{\partial B_1}u(ry)\nabla u(ry)\cdot y d\mathcal{H}.
\end{equation*}
This together with \eqref{u1}-\eqref{u4} implies
\begin{equation}\label{u5}
\begin{split}
e'(r)&=-6r^{-5}\int_{B_1}F(ry)dy+r^{-5}\int_{\partial B_1}F(ry)d\mathcal{H}+18r^{-7}\int_{\partial B_1}u^2(ry)d\mathcal{H}\\
&\quad -6r^{-6}\int_{\partial B_1}u(ry)\nabla u(ry)\cdot y d\mathcal{H}.
\end{split}
\end{equation}
On the other hand, by the definition of $e(r)$, one can see that
\begin{equation*}
e(r)=r^{-4}\int_{B_1}F(ry)dy-3r^{-6}\int_{\partial B_1}u^2(ry)d\mathcal{H}-\Phi(0+).
\end{equation*}
As a result, we conclude that
\begin{equation}\label{u6}
\begin{split}
e'(r)=&-\frac{6}{r}e(r)-\frac{6}{r}\Phi(0+)-6r^{-6}\int_{\partial B_1}u(ry)\nabla u(ry)\cdot y d\mathcal{H}\\
&\quad+r^{-5}\int_{\partial B_1}F(ry)d\mathcal{H}.
\end{split}
\end{equation}
Furthermore, by using $u_r(y)=\frac{u(ry)}{r^3}$, we obtain that
\begin{equation}\label{u7}
\begin{split}
e'(r)=&-\frac{6}{r}e(r)-\frac{6}{r}\Phi(0+)+\frac{1}{r}\int_{\partial B_1}\left(| \nabla u_r(y)|^2+2|y|u^+_r(y) \right)d\mathcal{H}\\
&-\frac{6}{r}\int_{\partial B_1} u_r(y)\nabla u_r(y)\cdot yd\mathcal{H}.
\end{split}
\end{equation}
Next, we define ${\nabla}^\perp_\nu:=\nabla u_r-(\nabla u_r\cdot\nu)\nu$ to denote the projection of $\nabla u_r$  onto the tangent plane orthogonal to $\nu$, where $\nu$ is the unit outward normal of $\partial B_1$. Therefore, it follows from \eqref{u7} that
\begin{equation}\label{u8}
\begin{split}
e'(r)=&-\frac{6}{r}(e(r)+\Phi(0+) )+\frac{1}{r}\int_{\partial B_1}\left( |{\nabla}^\perp_\nu u_r|^2+|\nabla u_r\cdot\nu |^2+2|y|u^+_r  \right)d\mathcal{H}\\
&-\frac{6}{r}\int_{\partial B_1}u_r\nabla u_r\cdot yd\mathcal{H}\\
=& -\frac{6}{r}(e(r)+\Phi(0+) )+ \frac{1}{r}\int_{\partial B_1}|{\nabla}^\perp_\nu u_r|^2 d\mathcal{H}+\frac{1}{r}\int_{\partial B_1} (\nabla u_r\cdot\nu-3u_r )^2d\mathcal{H}\\
&-\frac{9}{r}\int_{\partial B_1}{u_r}^2d\mathcal{H} +\frac{2}{r}\int_{\partial B_1}|y|u^+_rd\mathcal{H}\\
\ge& -\frac{6}{r}(e(r)+\Phi(0+) )+ \frac{1}{r}\int_{\partial B_1}|{\nabla}^\perp_\nu u_r|^2d\mathcal{H}+\frac{2}{r}\int_{\partial B_1}|y|u^+_rd\mathcal{H}\\
&-\frac{9}{r}\int_{\partial B_1}{u_r}^2d\mathcal{H}.
\end{split}
\end{equation}
Note that $c_r(y)$ is a homogeneous function of degree $3$ and $u_r= c_r$ on $\partial B_1$, hence it follows from \eqref{u8} that
\begin{equation}\label{u9}
\begin{split}
e'(r)\ge &-\frac{6}{r}(e(r)+\Phi(0+) )+\frac{1}{r}\int_{\partial B_1}\left( |{\nabla}^\perp_\nu c_r|^2+2|y|c^+_r+(3c_r)^2 \right)d\mathcal{H}-\frac{18}{r}\int_{\partial B_1}{c_r}^2d\mathcal{H}\\
=& -\frac{6}{r}(e(r)+\Phi(0+) )+\frac{1}{r}\int_{\partial B_1}\left( |{\nabla}^\perp_\nu c_r|^2+2|y|c^+_r+|\nabla c_r\cdot \nu|^2 \right)d\mathcal{H}-\frac{18}{r}\int_{\partial B_1}{c_r}^2d\mathcal{H}\\
=& -\frac{6}{r}(e(r)+\Phi(0+) )+\frac{1}{r}\int_{\partial B_1}\left(|\nabla c_r|^2+2|y|c^+_r  \right)d\mathcal{H}-\frac{18}{r}\int_{\partial B_1}{c_r}^2d\mathcal{H}.
\end{split}
\end{equation}
Then, by the homogeneous of $c_r$ again, we arrive at
\begin{equation*}
\begin{split}
\int_{B_1}\left(|\nabla c_r|^2+2|y|c^+_r  \right)dy=&\int_{0}^{1}d\rho\int_{\partial B_\rho} \left(|\nabla c_r|^2+2|y|c^+_r  \right)d\mathcal{H}\\
=&\int_{0}^{1}{\rho}^5d\rho \int_{\partial B_1}\left(|\nabla c_r|^2+2|z|c^+_r  \right)d\mathcal{H}\\
=&\frac{1}{6}\int_{\partial B_1}\left(|\nabla c_r|^2+2|z|c^+_r  \right)d\mathcal{H}.
\end{split}
\end{equation*}
Inserting this in \eqref{u9}, we infer that
\begin{equation}\label{u10}
\begin{split}
e'(r)\ge &-\frac{6}{r}(e(r)+\Phi(0+) )+\frac{6}{r}\int_{B_1}\left(|\nabla c_r|^2+2|y|c^+_r  \right)dy-\frac{18}{r}\int_{\partial B_1}{c_r}^2d\mathcal{H}\\
=&\frac{6}{r}( M(c_r)-e(r)-\Phi(0) ).
\end{split}
\end{equation}
On one hand, the epiperimetric inequality shows that
\begin{equation*}
M(v)\le (1-\kappa)M(c_r)+\kappa\Phi(0)
\end{equation*}
for some $v\in W^{1,2}(B_1)$ with $v=c_r$ on $\partial B_1$. On the other hand, $u$ is the minimizer of the energy $M(v)$. Hence we obtain
\begin{equation*}
M(u_r)\le M(v)\le (1-\kappa)M(c_r)+\kappa\Phi(0),
\end{equation*}
which together with \eqref{u10} and the definition of $M(u_r)$ yields that
\begin{equation*}
\begin{split}
e'(r)\ge& \frac{6}{r}\left[ \frac{1}{1-\kappa}(M(u_r)- \Phi(0) )-e(r)  \right]\\
=&\frac{6\kappa}{1-\kappa}\frac{e(r)}{r}.
\end{split}
\end{equation*}
This is equivalent to
\begin{equation*}
\left( r^{\frac{6\kappa}{\kappa-1}} e(r)   \right)'\ge0.
\end{equation*}
Then integrating this expression from $r$ to $r_0$ gives
\begin{equation}\label{u11}
e(r)\le \left( \frac{r}{r_0} \right)^{ \frac{6\kappa}{\kappa-1} }e(r_0) \quad\textrm{for all}\quad 0<r<r_0.
\end{equation}
This complete the proof of the first statement.\\

In what follows, we will show that the decay rate of the energy implies the uniqueness of the blowup limits.

For any $0<t<\sigma<r_0<1$, we have
\begin{equation}\label{u12}
\begin{split}
\int_{\partial B_1}\left|  \frac{u(\sigma x)}{{\sigma}^3}- \frac{u(t x)}{t^3} \right|d\mathcal{H}&\le \int_{\partial B_1}  \int_{t}^{\sigma} \left| \frac{d}{ds}u_s  \right|dsd\mathcal{H}\\
&= \int_{t}^{\sigma} \frac{1}{s^3} \int_{\partial B_1} \left| \nabla u(sx)\cdot x- \frac{3u(sx)}{s}\right|dsd\mathcal{H}.
\end{split}
\end{equation}
Using H\"{o}lder's inequality once more, one can see that
\begin{equation}\label{u13}
\begin{split}
\int_{t}^{\sigma} \frac{1}{s^3} \int_{\partial B_1} \left| \nabla u(sx)\cdot x- \frac{3u(sx)}{s}\right|dsd\mathcal{H} &\le C\int_{t}^{\sigma}s^{-\frac{1}{2}}\left( \frac{2}{s^5}\int_{\partial B_1} \left|\nabla u(sx)\cdot x-\frac{3u(sx)}{s}  \right|^2 d\mathcal{H}  \right)^{\frac{1}{2}}ds\\
&=C\int_{t}^{\sigma}s^{-\frac{1}{2}}(\Phi'(s))^{ \frac{1}{2} }ds\\
&=C\int_{t}^{\sigma}s^{-\frac{1}{2}}(e'(s))^{ \frac{1}{2} } ds\\
&\le C\left( \int_{t}^{\sigma}s^{-1}ds   \right)^{\frac{1}{2}}\left( \int_{t}^{\sigma} e'(s)  ds\right)^{\frac{1}{2}}\\
&=C\left( \log\sigma- \log t  \right)^{\frac{1}{2}}\left( e( \sigma)-e(t)  \right)^{\frac{1}{2}}.
\end{split}
\end{equation}
Thus, collecting \eqref{u12} and \eqref{u13}, one can find that for any $0<t<\sigma<r_0<1$ we have
\begin{equation}\label{u14}
\int_{t}^{\sigma} \int_{\partial B_1}\left|   \frac{d}{ds}u_s  \right|d\mathcal{H}ds\le C\left( \log\sigma- \log t  \right)^{\frac{1}{2}}\left( e( \sigma)-e(t)  \right)^{\frac{1}{2}}.
\end{equation}

Now, considering $\rho$ and  $r$ satisfying $0<2\rho<2r< r_0<1 $, it is easy to check that there are integers $K,J>0$ with $K<J$ such that
\begin{equation}\label{u15}
2^{-(K+1)}\le \rho<2^{-K},\quad
2^{-(J+1)}\le r<2^{-J}.
\end{equation}
And therefore, we have
\begin{equation*}
\begin{split}
\int_{\partial B_1}\left|  \frac{u(r x)}{{r}^3}- \frac{u(\rho x)}{{\rho}^3} \right|d\mathcal{H}&\le \int_{\rho}^{r} \int_{\partial B_1} \left|   \frac{d}{ds}u_s  \right|d\mathcal{H}ds\\
&\le \sum\limits_{i=J}^{K}  \int_{\frac{1}{2^{i+1}}}^{\frac{1}{2^i}} \int_{\partial B_1} \left|   \frac{d}{ds}u_s  \right|d\mathcal{H}ds,
\end{split}
\end{equation*}
which together  with \eqref{u14} yields that
\begin{equation}\label{u16}
\begin{split}
\int_{\partial B_1}\left|  \frac{u(r x)}{{r}^3}- \frac{u(\rho x)}{{\rho}^3} \right|d\mathcal{H}&\le C \sum\limits_{i=J}^{K}\left( \log{\frac{1}{2^i}}- \log{\frac{1}{2^{i+1}} }  \right)^{\frac{1}{2}}\left( e\Big( \frac{1}{2^i} \Big)-e\Big(\frac{1}{2^{i+1}}\Big)  \right)^{\frac{1}{2}}\\
&\le C \sum\limits_{i=J}^{K}\left(  e\Big( \frac{1}{2^i} \Big) \right)^{\frac{1}{2}}\\
&\le C \sum\limits_{i=J}^{K} \left(   e(r_0)\left( \frac{2^{-i}}{r_0} \right)^{ \frac{6\kappa}{1-\kappa} } \right)^{\frac{1}{2}}\\
&\le C(\kappa)\sum\limits_{i=J}^{\infty} q^i\\
&\le C(\kappa)\frac{q^J}{1-q}\\
&\le C(\kappa) r^{\frac{3\kappa}{1-\kappa}}
\end{split}
\end{equation}
due to the fact $\frac{1}{2^{J+1}}\le r$, where $q=\left( \frac{1}{2}  \right)^{\frac{3\kappa}{1-\kappa}}\in (0,1)$. To finish the proof, letting $u_{\rho_m}\to u_0$ as a certain sequence $\rho_m\to0$, one can arrive at
\begin{equation*}
\int_{\partial B_1}|u_r-u_0|d\mathcal{H}\le C(\kappa)r^{\frac{3\kappa}{1-\kappa}}.
\end{equation*}
This completes the proof of Theorem \ref{TH1.2un}.
\end{proof}

\section{Convergence of the free boundary} \label{sect Convergence of the free boundary}

In the following chapters, we want to ``transfer" some key information from  blowup limits to the original solutions.
To this end, the first step is to give the convergence of the free boundary to the free boundary of the blowup profiles. Before doing this, we present some important observations.\\

Any global solution with the energy $\frac{\pi}{81}$, up to a rotation,  is of the form
\begin{equation}\label{we1}
u^*=\frac{r^3}{9}(1-\sin3\theta)  \quad\textrm{for all}\quad \frac{\pi}{6}<\theta<\frac{5\pi}{6}.
\end{equation}
Meanwhile, all global solutions with the energy $\frac{\pi}{81}$  have a lower bound.
\begin{lem}\label{L6.1}
Let $u$ be any global solution of \eqref{g1} with $\phi(1,u)=\frac{\pi}{81}$. Then we have
\begin{equation}\label{con1}
u(x)\ge C  d^2(x,\{ u=0\})( |x|+ d(x,\{ u=0\}))  \quad\textrm{for all}\quad x\in \mathbb{R}^2,
\end{equation}
where the constant $C>0$ is independent of $u$.
\end{lem}
\begin{proof}
The statement is evident for $x\in\{u=0\}$. Let $x\in\{u>0\}$. Observe that both $u$ and $d(x,\{ u=0\})$ are homogeneous function, hence the statement \eqref{con1} is equivalent to
\begin{equation}\label{con2}
u(e^{i\theta})\ge C  d^2(e^{i\theta},\{ u=0\})( 1+ d(e^{i\theta},\{ u=0\}))  \quad\textrm{for all}\quad \theta_1<\theta<\theta_1+\frac{2\pi}{3}.
\end{equation}
For simplicity, we denote $d(e^{i\theta},\{ u=0\})$ as $d$. Define the free boundary $\Gamma_u:=\Gamma^1_u\cup \Gamma^2_u$, where the angles corresponding to $\Gamma^1_u$ and $\Gamma^2_u$ are $\theta_1$ and  $\theta_2$, respectively.
 Define $d_1:=dist(e^{i\theta}, \Gamma^1_u)$ and $d_2:=dist(e^{i\theta}, \Gamma^2_u)$. Then we have $d=\min\{ d_1,d_2 \}$.
It is easy to check that $d_1=|\sin\widetilde{\theta}|$ and $d_2=|\sin(\widetilde{\theta}-\frac{2\pi}{3})|$, where $\widetilde{\theta}=\theta-\theta_1\in (0,\frac{2\pi}{3})$.
Define
\begin{equation}\label{con3}
\varrho( \widetilde{\theta}):=\frac{1-\cos3\widetilde{\theta}}{9d^2(d+1)}.
\end{equation}
One can see that $0<\varrho( \widetilde{\theta})\in C(0,\frac{2\pi}{3})$.
On one hand, by L'Hospital's rule, it is easy to find $\lim_{\widetilde{\theta} \to 0+}\varrho( \widetilde{\theta})=\lim_{\widetilde{\theta} \to  \frac{2\pi}{3}-}\varrho( \widetilde{\theta}) =\frac{1}{2} $.
On the other hand, one can easy check $\varrho( \frac{\pi}{3})=\frac{8\sqrt{3}}{81}<\varrho(0)=\varrho(\frac{2\pi}{3})$. Therefore,  $\varrho( \widetilde{\theta})$
 must attain its minimum  within the interval $(0,\frac{2\pi}{3})$. Taking $C=\min_{0< \widetilde{\theta} <\frac{2\pi}{3}}\varrho( \widetilde{\theta})$, we obtain desired result.  Here $C$ is independent with $\theta_1$.
\end{proof}

In the following lemma, one can see that when $||u-u^* ||_{L^\infty (B_1)}$ is small, if $x$ is far from $\{u^*=0\}$, then $\{u(x)>0\}$; if $x$ is far from $\{u^*>0\}$, then $\{u(x)=0\}$.
\begin{lem}\label{L6.2}
Let $u$ satisfy \eqref{f6} in $B_1$. If there exists a small constant $\varepsilon>0$ such that
\begin{equation}\label{con4}
||u-u^* ||_{L^\infty (B_1)}\le \varepsilon,
\end{equation}
then we have
\begin{equation}\label{con5}
u>0 \quad in \, B_1\cap \left\{x: d^2(x,\{ u^*=0\})( |x|+ d(x,\{ u^*=0\}))> \varepsilon^{-1} ||u-u^* ||_{L^\infty (B_1)} \right\},
\end{equation}
\begin{equation}\label{con6}
B_{1/2}\cap \left\{x: d^2(x,\{ u^*>0\})( |x|+ d(x,\{ u^*>0\}))> \varepsilon^{-1} ||u-u^* ||_{L^\infty (B_1)} \right\}\subset (\{u=0\}\cap B_1)^\circ.
\end{equation}
\end{lem}
\begin{proof}
With the help of \eqref{con1}, one can find
\begin{equation}\label{con7}
u^*(x)\ge Cd^2(x,\{ u^*=0\})( |x|+ d(x,\{ u^*=0\}))  \quad\textrm{for all}\quad x\in \mathbb{R}^2,
\end{equation}
For $\varepsilon$ suitably small, by the assumption in \eqref{con5}, we have
\begin{equation*}
\begin{split}
u&=u^*+u-u^*\ge u^*-||u-u^*||_{L^\infty (B_1)}\\
&\ge Cd^2(x,\{ u^*=0\})( |x|+ d(x,\{ u^*=0\}))-||u-u^*||_{L^\infty (B_1)}\\
&> \left( \frac{C}{\varepsilon} -1 \right)||u-u^*||_{L^\infty (B_1)}.
\end{split}
\end{equation*}
This completes the proof of the first part.

Now, we claim that  if $y\in B_1$, $B_r(y)\subset\subset \{u^*=0 \} \cap B_1 $ and
\begin{equation}\label{con8}
c_1r^2(r+|y|)> ||u-u^*||_{L^\infty (B_1)} ,
\end{equation}

then we have $u(y)=0$, where the positive constant $c_1$ is from Lemma \ref{dege}.

Assume by contradiction that $u(y_0)>0$, then it follows from  Lemma \ref{dege} that
\begin{equation*}
\sup_{\{u>0 \}\cap\partial B_r(y_0)} u\ge u(y_0)+c_1r^2\left(r+|y_0|  \right).
\end{equation*}
Since $B_r(y_0)\subset\subset \{u^*=0 \}$, we obtain that
\begin{equation*}
\begin{split}
0&=\sup_{\{u>0 \}\cap\partial B_r(y_0)}u^*\ge \sup_{\{u>0 \}\cap\partial B_r(y_0)}u-||u-u^*||_{L^\infty (B_1)}\\
&\ge c_1r^2\left(r+|y_0|  \right)-||u-u^*||_{L^\infty (B_1)}\\
&>u(y_0)
\end{split}
\end{equation*}
due to the assumptions, and hence we reach a contradiction. For all $x_0\in \left( \{u^*=0 \} \cap B_1  \right)^\circ$, we  take  $r=\frac{1}{2}d(x_0, (\{u^*=0 \} \cap B_1)^c )$.
Then by the above claim, if $r$ satisfies \eqref{con8}, i.e.,
\begin{equation}\label{con9}
\begin{split}
\frac{c_1}{8}d^2(x_0, (\{u^*=0 \} \cap B_1)^c )(2|x_0|+d(x_0, (\{u^*=0 \} \cap B_1)^c  )>||u-u^* ||_{L^\infty (B_1)},
\end{split}
\end{equation}
then one can arrive at $u(x_0)=0$. On the other hand, it is easy to see that
\begin{equation*}
\begin{split}
\mathcal{S}:=\left\{x\in B_1:  \frac{c_1}{8} d^2(x,(\{u^*=0 \} \cap B_1)^c)( |x|+ d(x,(\{u^*=0 \} \cap B_1)^c))>||u-u^* ||_{L^\infty (B_1)} \right\}
\end{split}
\end{equation*}
suffices to deduce \eqref{con9}. Therefore, one can find $\mathcal{S}\subset\{u=0\}\cap B_1 $.
Actually,  one can further obtain
\begin{equation}\label{con10}
\mathcal{S}\subset (\{u=0\}\cap B_1)^\circ.
\end{equation}
Otherwise, if there exists $x_0\in \mathcal{S}\cap \Gamma$, then one can find a point $z\in B_{\varsigma}(x_0)$ with $\varsigma>0$ small enough such that $u(z)>0$.
On the other hand, by the continuity of the distance function $x \mapsto d(x,(\{u^*=0 \} \cap B_1)^c )$ and the smallness of $\varsigma$, one can find the point $z$ meets \eqref{con9}. Therefore, we have $u(z)=0$, a contradiction.

Now, for all $x\in B_{1/2}\cap \left\{x: d^2(x,\{ u^*>0\})( |x|+ d(x,\{ u^*>0\}))> \varepsilon^{-1} ||u-u^* ||_{L^\infty (B_1)} \right\}$, one can see that
\begin{equation}\label{con11}
\begin{split}
d(x,(\{u^*=0 \} \cap B_1)^c )&=\min\{d(x, \{u^*>0 \}), d(x,B^c_1 ) \}\\
& \ge \min\{d(x, \{u^*>0 \}), \frac{1}{2}\}.
\end{split}
\end{equation}
Furthermore,  with the aid of \eqref{con11}, by choosing $0<\varepsilon<<c_1$  we obtain
\begin{equation}\label{con12}
\begin{split}
\frac{c_1}{8}& d^2(x,(\{u^*=0 \} \cap B_1)^c)( |x|+ d(x,(\{u^*=0 \} \cap B_1)^c))\\
&\ge \min\left\{   \frac{c_1}{8} d^2(x, \{u^*>0 \})(|x|+d(x, \{u^*>0 \})), \frac{c_1}{32}\left(|x|+\frac{1}{2}\right)\right\}\\
&\ge \min\left\{ \varepsilon d^2(x, \{u^*>0 \})(|x|+d(x, \{u^*>0 \})), 2\varepsilon \right\}\\
&> ||u-u^*||_{L^\infty (B_1)}
\end{split}
\end{equation}
due to the assumption  \eqref{con4}.  Hence, by applying \eqref{con10}, we arrive that
\begin{equation*}
B_{1/2}\cap \left\{x: d^2(x,\{ u^*>0\})( |x|+ d(x,\{ u^*>0\}))> \varepsilon^{-1} ||u-u^* ||_{L^\infty (B_1)} \right\} \subset (\{u=0\}\cap B_1)^\circ
\end{equation*}
under the assumption \eqref{con4}, which leads to \eqref{con6}.
\end{proof}

Next, we show that if $u$ is very close to $u^*$, then  the free boundary of $u$ lies close to $\Gamma_{u^*}$.

\begin{lem}\label{L6.3}
Let $u$ be any solution of  \eqref{f6} in $B_1$. There exists a constant $\varepsilon>0$ such that if
\begin{equation*}
||u-u^* ||_{L^\infty (B_1)}\le \varepsilon,
\end{equation*}
then we have
\begin{equation}\label{con13}
\Gamma_u\cap B_{\frac{1}{2}}\subset \left\{ d^2(x,\Gamma_{u^*})(|x|+ d(x,\Gamma_{u^*}))\le \varepsilon^{-1}||u-u^* ||_{L^\infty (B_1)} \right\}.
\end{equation}
\end{lem}
\begin{proof}
If $u=u^*$, the statement is obvious. In the case $u\neq u^*$, assume by contradiction that there exists $x_0\in \Gamma_u\cap B_{\frac{1}{2}}$ such that
\begin{equation}\label{con14}
d^2(x_0,\Gamma_{u^*})(|x_0|+ d(x_0,\Gamma_{u^*}))> \varepsilon^{-1}||u-u^* ||_{L^\infty (B_1)}.
\end{equation}
In view of $d(x_0,\Gamma_{u^*})=\max \left\{ d(x_0,{u^*=0}),d(x_0,{u^*>0}) \right\}$, which together with \eqref{con14} implies that either of the following holds\\[-12pt]
\begin{flalign*}
&\quad(\romannumeral 1)\quad d^2(x_0,\{u^*=0\})(|x_0|+ d(x_0,\{u^*=0\}))> \varepsilon^{-1}||u-u^* ||_{L^\infty (B_1)}, &\\
&\quad(\romannumeral 2)\quad d^2(x_0,\{u^*>0\})(|x_0|+ d(x_0,\{u^*>0\}))> \varepsilon^{-1}||u-u^* ||_{L^\infty (B_1)}. &
\end{flalign*}
For the case $(\romannumeral 1)$, by Lemma \ref{L6.2}, we have $u(x_0)>0$, a contradiction. In the case of $(\romannumeral 2)$, by Lemma \ref{L6.2} again, one can see that
$x_0\in \{u=0\}^\circ$, which contradicts $x_0\in \Gamma_u$.
\end{proof}

With Lemma \ref{L6.3} in hand, one can show that the free boundary of original solution $u$ converges to the free boundary of $u^*$.
\begin{lem}\label{L6.4}
Let $u$ be any solution of \eqref{f6} in $B_1$, and $0$ is a regular point. If $x_0\in \Gamma_u$ with $|x_0|$ small enough, then there exists a constant $C>0$ such that
\begin{equation}\label{con15}
d(x_0,\Gamma_{u^*})\le C|x_0|  \|u_{4|x_0|}-u^* \|^{\frac{1}{3}}_{C^0 (B_1)}.
\end{equation}
\end{lem}
\begin{proof}
Fix $x_0\in \Gamma_u$, and assume $|x_0|$ is small enough. It is easy to see that $u_{4|x_0|}$ solves \eqref{f6} and \eqref{con4}.
Hence by Lemma \ref{L6.3}  we have
\begin{equation}\label{con16}
\Gamma_{u_{4|x_0|}}\cap B_{\frac{1}{2}}\subset \left\{ d^2(z,\Gamma_{u^*})(|z|+ d(z,\Gamma_{u^*}))\le \varepsilon^{-1}||u_{4|x_0|}-u^* ||_{L^\infty (B_1)} \right\}.
\end{equation}
On  one hand,   $\left\{u^*>0\right\}$ is a cone, hence we have
\begin{equation}\label{con17}
d(\lambda z, \Gamma_{u^*})=\lambda d(z,\Gamma_{u^*}) \quad\textrm{for all}\quad  \lambda>0.
\end{equation}
On the other hand, one can find
\begin{equation}\label{con18}
\Gamma_{u_{4|x_0|}}\cap B_{\frac{1}{2}}=\frac{1}{4|x_0|} \Gamma_{u}\cap B_{2|x_0|}.
\end{equation}
Thus, putting  \eqref{con16}-\eqref{con18} together, we obtain
\begin{equation}\label{con19}
\begin{split}
\Gamma_{u}\cap B_{2|x_0|}&\subset  \left\{4|x_0|z:  d^2(z,\Gamma_{u^*})\left(|z|+ d(z,\Gamma_{u^*})\right)\le \varepsilon^{-1}||u_{4|x_0|}-u^* ||_{L^\infty (B_1)} \right\}\\
&= \left\{x:  d^2\left(\frac{x}{4|x_0|} ,\Gamma_{u^*}\right)\left(\frac{|x|}{4|x_0|}+ d(\frac{x}{4|x_0|},\Gamma_{u^*})\right)\le \varepsilon^{-1}||u_{4|x_0|}-u^* ||_{L^\infty (B_1)} \right\}\\
&= \left\{x:  d^2(x ,\Gamma_{u^*})\left(|x|+ d(x,\Gamma_{u^*})\right)\le 64|x_0|^3\varepsilon^{-1}||u_{4|x_0|}-u^* ||_{L^\infty (B_1)} \right\}.
\end{split}
\end{equation}
Furthermore, since  $x_0\in \Gamma_u$, it follows from \eqref{con19} that
\begin{equation*}
d^2(x_0 ,\Gamma_{u^*})\left(|x_0|+ d(x_0,\Gamma_{u^*})\right)\le 64|x_0|^3\varepsilon^{-1}||u_{4|x_0|}-u^* ||_{L^\infty (B_1)},
\end{equation*}
which implies \eqref{con15}, where $C=4\varepsilon^{-\frac{1}{3}}$.
\end{proof}

\section{The weak directional monotonicity } \label{sect The weak directional monotonicity }

Based on a geometric property of a weak directional monotonicity of the global solution, one can give a similar property for the solution of \eqref{f6} near a usual regular point.
In classical obstacle problem, a directional monotonicity of a half-space solution is obvious. However, in our setting, the structure of the cubic homogeneous global solutions is much more complex.
In general, we cannot expect such a strong property. \\

Now, we are in the position to give a weak directional monotonicity of the global solution $u$ to \eqref{g1} with the energy $\Phi(1,u)=\frac{\pi}{81}$.
Denote by $e^*=\left(-\frac{1}{2},\frac{\sqrt{3}}{2}  \right)$ the the unit normal vector of $u^*$ pointing towards its positive set in the first quadrant.
\begin{lem}\label{L7.1}
Suppose $e=(-\sin(\frac{\pi}{6}+\gamma),\cos(\frac{\pi}{6}+\gamma))$ with $\gamma\in(-\frac{\pi}{2}, \frac{\pi}{2})$, $c_0>0$, $\delta>0$, and $\epsilon>0$. If
$e\cdot e^*\ge\frac{c_0+1}{2c_0}\epsilon$, then we have
\begin{equation}\label{we2}
c_0\partial_eu^*-u^*\ge0 \quad\textrm{in}\quad B_1\verb+\+\{0 \}   \cap \left\{\frac{\pi}{6}<\theta<\frac{\pi}{6}+\epsilon \right\}.
\end{equation}
\end{lem}
\begin{proof}
By differentiating the identity \eqref{we1}  along the direction $e$, we obtain
\begin{equation*}
\begin{split}
\partial_eu^*&= \frac{r^2}{3} \left( - \sin\left(\frac{\pi}{6}+\gamma \right)(\cos\theta-\sin2\theta)+\cos\left( \frac{\pi}{6}+\gamma \right)(\sin\theta-\cos2\theta  )  \right)\\
&=\frac{2r^2}{3}\left( \sin\left(\frac{\pi}{6}+\gamma \right)\sin\frac{ \frac{\pi}{3} -\phi}{2} +\cos\left(\frac{\pi}{6}+\gamma \right)\cos\frac{ \frac{\pi}{3} -\phi}{2} \right)\sin\frac{3\phi}{2}\\
&=\frac{2r^2}{3}\sin\frac{3\phi}{2}\cos\left(\gamma+\frac{\phi}{2}  \right).
\end{split}
\end{equation*}
where $\phi=\theta-\frac{\pi}{6}$. Therefore, we have
\begin{equation}\label{we3}
c_0\partial_eu^*-u^*=\frac{2r^2}{3}\sin\frac{3\phi}{2}\left( c_0\cos\left(\gamma+\frac{\phi}{2}  \right)-\frac{r}{3}\sin\frac{3\phi}{2}            \right).
\end{equation}
Since $\frac{\pi}{6}<\theta<\frac{5\pi}{6}$, one have $0<\phi<\frac{2\pi}{3}$, which implies  $\sin\frac{3\phi}{2}>0$. And hence, to ensure $c_0\partial_eu^*-u^*\ge0$,
it suffices to show that
\begin{equation}\label{we4}
\sin\frac{3\phi}{2}\le3 c_0 \cos\left(\gamma+\frac{\phi}{2}  \right).
\end{equation}
On one hand, we have
\begin{equation}\label{we5}
\frac{1}{3}\sin\frac{3\phi}{2}\le \frac{\phi}{2} <\frac{\epsilon}{2}
\end{equation}
provided that $0<\phi=\theta-\frac{\pi}{6}< \epsilon$. On the other hand,
\begin{equation}\label{we6}
\cos\left(\gamma+\frac{\phi}{2}  \right)\ge \cos\gamma-\frac{\phi}{2}> \cos\gamma-\frac{\epsilon}{2}.
\end{equation}
Combining \eqref{we5}, \eqref{we6}, and the assumption $e\cdot e^*\ge\frac{c_0+1}{2c_0}\epsilon$, we arrive at \eqref{we4}.
This completes the proof of this lemma.
\end{proof}

Next, we give the following important lemma.
\begin{lem}\label{L7.2}
Let $u$ be any solution to \eqref{f6} in $B_1$. Suppose $z\in \{u>0\}$, $B_r(z)\subset\subset B_1\cap \{|x|\ge \frac{1}{10} \}$. If
\begin{equation}\label{we00}
u(z)-\frac{1}{20}\partial_e u(z)>0,
\end{equation}
then we have
\begin{equation}\label{we7}
\frac{r^2}{80}\le \sup_{\{u>0 \}\cap \partial B_r(z) }\left( u-\frac{1}{20}\partial_e u \right).
\end{equation}
\end{lem}
\begin{proof}
Consider  the auxiliary function
\begin{equation}\label{we8}
\Upsilon(x)=u(x)-\frac{1}{20}\partial_e u(x)-\left(  u(z)-\frac{1}{20}\partial_e u(z)   \right)-\frac{|x-z |^2}{80}.
\end{equation}
First, it is easy to see that $\Upsilon(z)=0$ and
\begin{equation}\label{we9}
\Upsilon(x)\le-\left(  u(z)-\frac{1}{20}\partial_e u(z)   \right) \quad\textrm{on}\quad x\in \Gamma_u
\end{equation}
thanks to the assumption. On the other hand, by a simple calculation, we have
\begin{equation*}
\Delta \Upsilon(x)=|x|-\frac{1}{10}  \quad\textrm{in}\quad \{u>0\} \cap \left\{|x|\ge \frac{1}{10} \right\}.
\end{equation*}
In particular, thanks to $B_r(z)\subset\subset B_1\cap \{|x|\ge \frac{1}{10} \}$,  we have
\begin{equation*}
\Delta \Upsilon(x)\ge0 \quad\textrm{in}\quad \{u>0\} \cap B_r(z).
\end{equation*}
Therefore by the maximum principle we obtain that
\begin{equation}\label{we10}
0=\Upsilon(z)\le \sup_{ \partial(\{u>0 \}\cap  B_r(z) ) } \Upsilon(x)
\end{equation}
On the other hand, we have $ \partial(\{u>0 \}\cap  B_r(z) )=( \partial \{u>0 \}\cap B_r(z) )\cup( \{u>0 \}\cap \partial B_r(z) ):=S_1\cup S_2$.
This together with \eqref{we10} yields
\begin{equation*}
0\le \max\left\{ \sup_{S_1}\Upsilon(x), \sup_{S_2}\Upsilon(x) \right\}\le \max\left\{ \sup_{\Gamma_u}\Upsilon(x)  , \sup_{S_2}\Upsilon(x) \right\}.
\end{equation*}
Note that \eqref{we9} implies  $\sup_{\Gamma_u}\Upsilon(x) \le -\left(  u(z)-\frac{1}{20}\partial_e u(z)   \right)<0 $, hence one can arrive at
\begin{equation*}
0\le\sup_{S_2}\Upsilon(x)=-\left( u(z)-\frac{1}{20}\partial_e u(z)+\frac{r^2}{80}  \right)+\sup_{\{u>0 \}\cap \partial B_r(z) }\left( u(x)-\frac{1}{20}\partial_e u(x) \right).
\end{equation*}
This together with  \eqref{we00} yields the desired conclusion.
\end{proof}

In the next lemma we show that $C^1$ approximations by the global solution $u^*$ imply the directional monotonicity.
\begin{lem}\label{L7.3}
Let $u$ be any solution to \eqref{f6} in $B_1$. Suppose $x_0\in \Gamma_u\cap\partial B_{\frac{1}{4}}\cap \{x_1>0\}$ and $e=(-\sin(\frac{\pi}{6}+\gamma),\cos(\frac{\pi}{6}+\gamma))$ with $\gamma\in (-\frac{\pi}{2},\frac{\pi}{2})$, and $r>0$.
There exists a small constant $c>0$  such that if
\begin{equation}\label{we11}
||u-u^* ||^{\frac{1}{3}}_{C^1 (B_1)}+r\le ce\cdot e^*,
\end{equation}
then we have
\begin{equation}\label{we12}
\frac{1}{20}\partial_e u-u\ge0 \quad\textrm{in}\quad  B_r(x_0).
\end{equation}
\end{lem}
\begin{proof}
The statement is evident provided that $x\in \{u=0\}$, so we next assume that $x\in\{u>0\}\cap B_r(x_0)$.
In the case $u\equiv u^*$, the statement is obvious. Now, assume that $u\not\equiv u^*$.
Arguing by contradiction, suppose the statement fails  for some $y\in B_r(x_0)\cap \{u> 0\}$, and let $x_{u^*}=\Gamma_{u^*}\cap\partial B_{\frac{1}{4}}\cap \{x_1>0\}$.It is easy to see that
$x_{u^*}= \left( \frac{\sqrt{3}}{8},  \frac{1}{8}  \right)$.\\

\textbf{Step 1.} If $\eta\in \left(0, \frac{1}{40}  \right)$ and $\arcsin4\eta\le \frac{2}{21}e\cdot e^*$, then we have $\frac{1}{20}\partial_e u^*-u^*\ge0 \,\ in\, B_{\eta}(x_{u^*})$. \\

\noindent For any $x\in B_{\eta}(x_{u^*})$ with $\eta\in \left(0, \frac{1}{40} \right)$, one can deduce that $x_1>x_{(u^*)_1}-\eta>\frac{1}{10}$, which yields $B_{\eta}(x_{u^*})\subset \left\{|x|\ge \frac{1}{10}   \right\}$. Hence by applying Lemma \ref{L7.1} with $c_0=\frac{1}{20}$,
one can find $\frac{1}{20}\partial_e u^*-u^*\ge0 \,\ in\, B_{\eta}(x_{u^*})$ provided that $\arcsin4\eta\le \frac{2}{21}e\cdot e^*$.
\\

\textbf{Step 2.} If $c>0$ is small, $0<\lambda<\frac{1}{120}$, $x\in B_{\frac{1}{120}}(x_0)\cap \{u>0\}$ with $\frac{1}{20}\partial_e u(x)-u(x)<0$, then we have
\begin{equation}\label{we13}
\frac{\lambda^2}{80}\le \sup_{\{u>0 \}\cap \partial B_\lambda(x) }\left( u-\frac{1}{20}\partial_e u \right).
\end{equation}

\noindent If  $c>0$ is small enough, then we have $ ||u-u^* ||_{L^\infty (B_1)}<< \varepsilon$.  With the aid of Lemma \ref{L6.3},
one can derive that
\begin{equation*}
\Gamma_u\cap B_{\frac{1}{2}}\subset \left\{ d^2(x,\Gamma_{u^*})(|x|+ d(x,\Gamma_{u^*}))\le \varepsilon^{-1}||u-u^* ||_{L^\infty (B_1)} \right\}.
\end{equation*}
Hence, for $x_0\in \Gamma_u\cap\partial B_{\frac{1}{4}}\cap \{x_1>0\}$, we obtain
\begin{equation}\label{we14}
 d^2(x_0,\Gamma_{u^*})(|x_0|+ d(x_0,\Gamma_{u^*}))\le \varepsilon^{-1}||u-u^* ||_{L^\infty (B_1)}.
\end{equation}

Now, we claim that
\begin{equation}\label{we15}
|x_0-x_{u^*}|\le 2d(x_0,\Gamma_{u^*}).
\end{equation}
For any $x\in\partial B_{\frac{1}{4}}\cap  \{x_1>0\}$, we have $x=\left( \frac{\cos\theta}{4},\frac{\sin\theta}{4} \right)$ with $\theta\in [0, \frac{\pi}{2})\cup (\frac{3\pi}{2},2\pi )$.
By a simple calculation, we have
\begin{equation}\label{we16}
|x-x_{u^*}|=\frac{\sqrt{2}}{4} \sqrt{1-\cos(\theta-\frac{\pi}{6})}.
\end{equation}
We divide the proof into two cases.  Case $(\romannumeral 1)$: $\theta\in [0, \frac{\pi}{2})\cup [\frac{5\pi}{3},2\pi  )$, we deduce that $d(x,\Gamma_{u^*})=\frac{1}{4}|\sin(\theta-\frac{\pi}{6})|$. Note that $\theta-\frac{\pi}{6}\in (-\frac{\pi}{3},\frac{\pi}{2})\cup(\frac{3\pi}{2},2\pi)$, and hence  $\cos(\theta-\frac{\pi}{6} )\ge0$. This derive that
\begin{equation}\label{we17}
\frac{|x-x_{u^*}|}{d(x,\Gamma_{u^*}) }=\frac{\sqrt{2}\sqrt{1-\cos( \theta-\frac{\pi}{6})} }{|\sin(\theta-\frac{\pi}{6})|}=\frac{\sqrt{2}}{\sqrt{1+\cos(\theta-\frac{\pi}{6})}}\le \sqrt{2},
\end{equation}
which implies \eqref{we15} by choosing $x=x_0$. Case $(\romannumeral 2)$: $\theta\in(\frac{3\pi}{2}, \frac{5\pi}{3} )$. On one hand, we have $d(x,\Gamma_{u^*})=\frac{1}{4}$. On the other hand, by \eqref{we16}, we have $|x-x_{u^*}|\le \frac{1}{2}$. Thus, we have $\frac{|x-x_{u^*}|}{d(x,\Gamma_{u^*})}\le 2$, which derives that \eqref{we15} by choosing $x=x_0$.  Combining \eqref{we14} and \eqref{we15}, we have
\begin{equation}\label{wwee18}
|x_0-x_{u^*}|\le c_1 ||u- u^* ||^{\frac{1}{3}}_{L^\infty (B_1)}
\end{equation}
for some $c_1>0$, which together the smallness assumption on $c$ yields
$|x_0-x_{u^*}|<\frac{1}{120} $. So we have $(x_0)_1>(x_{u^*})_1-\frac{1}{120} $. This implies $B_{\frac{1}{120}}(x_0)\subset \{x_1> (x_{u^*})_1-\frac{1}{60}   \} $. For every $x\in B_{\frac{1}{120}}(x_0)$, we have
\begin{equation}\label{we18}
B_{\frac{1}{120}}(x)\subset \{x_1> (x_{u^*})_1-\frac{1}{40} \} \subset \{|x|\ge \frac{1}{10} \}.
\end{equation}
Therefore, if $\lambda\in (0, \frac{1}{120})$, it follows from \eqref{we18} that $B_{\lambda}(x)\subset \{|x|\ge \frac{1}{10} \}$. Furthermore, by using Lemma \ref{L7.2}, we prove \eqref{we13}.\\

\textbf{Step 3.} If $c>0$ is small, then we have $\frac{1}{20}\partial_e u-u\ge0 \,\ in\, B_r(x_0)$.

\noindent  For $r\in(0, \frac{1}{120})$, we have $y\in B_r(x_0)\cap \{u>0\}$ with $\frac{1}{20}\partial_e u(y)-u(y)<0$.
Taking
\begin{equation}\label{wwee19}
r_0=4\sqrt{5}||u-u^* ||^{\frac{1}{3}}_{C^1 (B_1)},
\end{equation}
we have $0<r_0<\frac{1}{120}$ provided that $c$ is small. Therefore, by step 2, we infer that
\begin{equation}\label{we19}
\frac{{r_0}^2}{80}\le \sup_{\{u>0 \}\cap \partial B_{r_0}(y) }\left( u-\frac{1}{20}\partial_e u \right).
\end{equation}
Define
\begin{equation}\label{we20}
\eta:= r_0+r+|x_0-x_{u^*}|.
\end{equation}
If $c$ is small enough, by  \eqref{we11}, \eqref{wwee18}, \eqref{wwee19},  one can deduce that $\eta\in (0, \frac{1}{40})$ is also small enough. So, we can infer that $\arcsin4\eta\le 2\pi\eta$. This implies that
\begin{equation*}
\frac{21}{2}\arcsin4\eta\le 21\pi(r_0+r+|x_0-x_{u^*} | )\le e\cdot e^*
\end{equation*}
provided that $c$ is small. Thus, by step1, we have $ \frac{1}{20}\partial_e u^*-u^*\ge0 $ in $B_{\eta}(x_{u^*})$. Furthermore, one can derive that $B_{r_0}(y)\subset B_{r_0+|x_0-y|}(x_0)\subset B_{r_0+|x_0-y|+|x_{u^*}-x_0|}(x_{u^* }) \subset B_\eta $. Hence we have
\begin{equation}\label{we21}
u^*-\frac{1}{20}\partial_e u^*\le0  \quad\textrm{in}\quad B_{r_0}(y).
\end{equation}
With the help of \eqref{we19} and \eqref{we21}, we infer that
\begin{equation*}
\begin{split}
\frac{{r_0}^2}{80}&\le \sup_{\{u>0 \}\cap \partial B_{r_0}(y) }\left( u-\frac{1}{20}\partial_e u \right)\le \sup_{\{u>0 \}\cap \partial B_{r_0}(y) }\left( u^*-\frac{1}{20}\partial_e u^* \right)\\
&\quad+\sup_{\{u>0 \}\cap \partial B_{r_0}(y)}\left(u- u^*-\frac{1}{20}\partial_e(u- u^* )\right)\\
&\le || u- u^*||_{C^1 (B_1)},
\end{split}
\end{equation*}
which implies $ r_0\le 4\sqrt{5}||u-u^* ||^{\frac{1}{2}}_{C^1 (B_1)}$. This contradicts \eqref{wwee19}.
\end{proof}

Now, using the previous lemma, we can show that there is a cone of directions $e$
in which the solution is monotone near the free boundary point $x_0$. Before giving the proof, we give some definitions. \\

Given a constant $\delta\in (0,1)$, we define the open cone and its dual cone as follows:
\begin{equation*}
C_{\delta}:=\left\{x\in \mathbb{R}^2: x\cdot e^*>\delta|x| \right\} \quad and \quad  C^*_{\delta}:=\left\{x\in \mathbb{R}^2: x\cdot y\ge0  \,\ \textrm{for all} \,\  y\in C_{\delta} \right\}
\end{equation*}

\begin{lem}\label{L7.4}
Let $u$ be any solution to \eqref{f6}, and suppose $x_0\in \Gamma_u\cap\partial B_{\frac{1}{4}}\cap \{x_1>0\}$. If $r>0$, $\delta\in(0,1)$, and $||u-u^* ||^{\frac{1}{3}}_{C^1 (B_1)}+r\le c\delta$ with $c$ as in Lemma \ref{L7.3},  then
\begin{equation}\label{L743}
u=0 \quad in\quad \Theta^{1}_{x_0}:= \{x\in B_r(x_0):x=x_0-te \,\ with \,\ e\in C_\delta,\, t>0\},
\end{equation}
\begin{equation}\label{L744}
u>0 \quad in\quad \Theta^{2}_{x_0}:=\{x\in B_r(x_0):x=x_0+te \,\ with \,\ e\in C_\delta,\, t>0\}.
\end{equation}
\end{lem}
\begin{proof}
 By Lemma \ref{L7.3}, we have
\begin{equation}\label{we23}
\partial_e u\ge0 \,\ in \,\  B_r(x_0) \quad\textrm{for all}\quad  e\in C_\delta.
\end{equation}
Since $u(x_0)=0$, it follows from \eqref{we23} and the non-degeneracy of $u$  that $u(x_0-te)=0$ for all $t>0$ and $e\in C_\delta$, which implies that $u=0$ in $\Theta^{1}_{x_0}$.  On the other hand, if there exists $z\in \Theta^{2}_{x_0}$  such that $u(z)=0$, then one can find $u(x)=0$
in $\{x\in B_r(x_0): x=z-te \,\ with\,\ e\in C_\delta,\,\ t>0  \}$. Now, repeating the same argument, we have $x_0\in \{u=0\}^\circ$, which is contradict $x_0\in \Gamma_u$.
\end{proof}

As a sequence of the previous lemma, we have the following fact.
\begin{lem}\label{L7.5}
Let $u$ be any solution to \eqref{f6}. Suppose $0$ is a regular point. Then there exists a constant $\rho>0$ such that all points on $\Gamma_u\cap B_{\rho}\cap \{x_1>0 \}$ are usual regular free boundary points.
\end{lem}
\begin{proof}
 Case $(\romannumeral 1)$: There exists $\rho>0$ such that $u= u^*$ in $B_\rho$.

\noindent Then  for any  $x_0\in \Gamma_u\cap B_\rho\cap \{x_1>0 \}$, by Caffarelli's classical results \cite{Caf1}, we have the blowup limit at $x_0$ is a halfspace  solution, which means $x_0$ is a usual regular point.

Case $(\romannumeral 2)$: $u\neq u^*$ in $B_\rho$ for all $\rho>0$.

\noindent In this case we have
\begin{equation}\label{we24}
||u_\rho-u^* ||_{C^1 (B_1)}>0.
\end{equation}
Since $||u_\rho-u^* ||_{C^1 (B_1)}\to 0$ as $\rho\to0$, for any $\sigma>0$ there exists  $\rho(\sigma)>0$ such that $||u_\rho-u^* ||_{C^1 (B_1)}\le \sigma$ provided that $\rho<4\rho(\sigma)$. For $x_0\in \Gamma_u\cap B_\rho\cap \{x_1>0 \}$, by choosing $\sigma=(\frac{c}{3})^3$, $\rho=4|x_0|$, then we have
\begin{equation}\label{we25}
||u_{4|x_0|}-u^* ||^{\frac{1}{3}}_{C^1 (B_1)}\le \frac{c}{3}
\end{equation}
with $|x_0|<\rho(\sigma)$. Let $y_0=\frac{1}{4|x_0|}x_0$, then we have $y_0\in \Gamma_{u_{4|x_0|}}\cap \partial B_{\frac{1}{4}}\cap \{x_1>0 \}$. Taking $r=||u_{4|x_0|}-u^* ||^{\frac{1}{3}}_{C^1 (B_1)}$, and $\delta=\frac{3}{c}||u_{4|x_0|}-u^* ||^{\frac{1}{3}}_{C^1 (B_1)}$, with the help of \eqref{we24} and \eqref{we25}, one can derive that $||u_{4|x_0|}-u^* ||^{\frac{1}{3}}_{C^1 (B_1)}+r<c\delta  $. Applying Lemma \ref{L7.4}, we obtain \begin{equation*}
\begin{split}
u_{4|x_0|}&=0 \quad in\quad \Theta^{1}_{y_0}:= \{y\in B_r(y_0):y=y_0-te \,\ with \,\ e\in C_\delta,\, t>0\},\\
u_{4|x_0|}&>0 \quad in\quad \Theta^{2}_{y_0}:=\{y\in B_r(y_0):y=y_0+te \,\ with \,\ e\in C_\delta,\, t>0\}.
\end{split}
\end{equation*}
Then rescaling back to $u$, we have
\begin{equation}\label{we26}
u=0 \quad in\quad \Theta^{1}_{x_0}:=\{x\in B_{4|x_0|r}(x_0):x=x_0-4|x_0|te \,\ with \,\ e\in C_\delta,\, t>0\},
\end{equation}
\begin{equation}\label{we27}
u>0 \quad in\quad \Theta^{2}_{x_0}:=\{x\in B_{4|x_0|r}(x_0):x=x_0+4|x_0|te \,\ with \,\ e\in C_\delta,\, t>0\}.
\end{equation}
If $x_0$ is a usual singular point, Caffarelli \cite{Caf1,Caf2} showed the blowup profile at $x_0$ is a nonnegative homogenous  quadratic polynomial denoted by $p_{x_0}(x)$. Then with the aid of homogeneity and \eqref{we26}, we have $p_{x_0}(x)=0$ in $x_0-C_\delta$  which is nonempty open cone. Furthermore, we have $x_0-C_\delta\subset \ker(D^2 p_{x_0}(x))$. This yields $\dim( \ker(D^2 p_{x_0}(x))  )\ge2$. Hence we have $p_{x_0}(x)\equiv0$ in $\mathbb{R}^2$, a contradiction.
\end{proof}

In the next lemma,  we give the convergence of the normal of the free boundary to the normal of the free boundary of the blowup limit at regular points.
\begin{lem}\label{L7.6}
Let $u$ be any solution to \eqref{f6}. Suppose $0$ is a regular point with blowup $u^*$. There exist $\rho>0$ and $C>0$  such that if  $x_0\in \Gamma_u\cap B_{\rho}\cap \{x_1>0 \}$, then we have
\begin{equation}\label{}
|e_{x_0}-e^*|\le C||u_{4|x_0|}-u^* ||^{\frac{1}{3}}_{C^1 (B_1)},
\end{equation}
where $e_{x_0}$ is the  normal to $\Gamma_u$ at $x_0$.
\end{lem}
\begin{proof}
 With the aid of Lemma \ref{L7.5}, we have any  $x_0\in \Gamma_u\cap B_{\rho}\cap \{x_1>0 \}$ is a usual regular  point. Therefore, $\Gamma_u$ has a tangent line at all these points. Applying \eqref{we26} and \eqref{we27}, one can find $e_{x_0}\cdot e^*\ge0$ for all $e\in C_\delta$ with $\delta=\frac{3}{c}||u_{4|x_0|}-u^* ||^{\frac{1}{3}}_{C^1 (B_1)}$. This implies that $e_{x_0}$ is from the dual of $C_{\delta}$, i.e., $e_{x_0}\in \overline{C_{\sqrt{1-\delta^2}}}$. Hence we have $e_{x_0}\cdot e^*\ge \sqrt{1-\delta^2}$. So, one can infer that
\begin{equation*}
\begin{split}
|e_{x_0}- e^*|= \sqrt{2-2e_{x_0}\cdot e^*  }\le \sqrt{2}\sqrt{1-\sqrt{1-\delta^2}}\le \sqrt{2}\delta.
\end{split}
\end{equation*}
This completes the proof.
\end{proof}

\section{Free boundary as a graph near regular points} \label{sect Free boundary as a graph near regular points}

Next, we will show that the free boundary is a  Lipschitz graph near the regular point.

\begin{lem}\label{L8.1}
Let $u$ be any solution to \eqref{f6}, and suppose $0$ is a regular point. There exist $\varrho>0$ such that if  $x_1\in [0,\frac{\varrho}{4})$, then there exists a unique $x_2$ satisfying $(x_1,x_2)\in \Gamma_u\cap B_{\varrho}$.
\end{lem}
\begin{proof}
Case $(\romannumeral 1)$: $x_1=0$. \\
\noindent For any $x=(0,x_2)\in B_{\varrho}\cap \{x_2>0 \}$, taking $\widetilde{x}=\frac{x}{2|x|}$, we have $d(\widetilde{x}, \{u^*= 0\})=\frac{\sqrt{3}}{4}$.  Hence, one can see that if $\varrho$  is small enough, we obtain \begin{equation*}
\varepsilon^{-1}|| u_{2|x|}-u^* ||_{L^\infty (B_1)}<d^2(\widetilde{x}, \{u^*= 0\})(|\widetilde{x}|+d(\widetilde{x}, \{u^*= 0\}))=\frac{6+3\sqrt{3}}{64},
\end{equation*}
where $\varepsilon>0$ is as in Lemma \ref{L6.2}. By using \eqref{con5}, we have $\widetilde{x}\in \{u_{2|x|}>0 \}$. Rescaling back to $u$, we have $u(x)>0$. Similarly, for any $x=(0,x_2)\in B_{\varrho}\cap \{x_2<0 \}$, taking $\widehat{x}=\frac{x}{4|x|}$, we have $d(\widehat{x}, \{u^*> 0\})=\frac{1}{4}$, and hence by \eqref{con6} we obtain $\widehat{x}\in \{u_{4|x|}= 0\}^\circ$, which implies $x\in \{u= 0\}^\circ$. Therefore, for $x_1=0$, there exists a unique $x_2=0$ such that $x=(0,x_2)\in \Gamma_u\cap B_{\varrho}$. \\
\indent Case $(\romannumeral 2)$: $0<x_1<\frac{\varrho}{4}$. We split the proof into two steps, namely the existence and uniqueness of $x_2$.\\
\indent \textbf{Claim 1.}  For any  $x_1\in (0,\frac{\varrho}{4})$, there exists a $x_2$ such that $x=(x_1,x_2)\in \Gamma_u\cap B_{\varrho}$. \\
Fix  $x_1\in(0, \frac{\varrho}{4})$, using the continuity of the $u(x_1,x_2)$ with respect to $x_2$,  it suffices to show that there exist two constants $s_1< s_2$ with $(x_1,s_1 ), (x_1,s_2 )\in B_{\varrho}$ such that $(x_1,s_1 )\in {\{u= 0\}}^\circ$ and $u(x_1,s_2 )>0$. Now, taking $\overline{x}=(\frac{x_1}{\varrho}, \frac{3}{4})$ and $\underline{x}=(\frac{x_1}{\varrho}, -\frac{1}{4})$.  It is easy to see that $|\overline{x}|<1, |\underline{x}|<\frac{1}{2}$. By a simple calculation, we obtain $d(\overline{x}, {\{u^*= 0\}} )=\frac{\sqrt{3}}{2}|\frac{\sqrt{3}x_1}{3\varrho}-\frac{3}{4}|> \frac{3\sqrt{3}-1}{8}$. Hence if $\varrho$ is small enough, we have $\varepsilon^{-1} d^2(\overline{x}, {\{u^*= 0\}} )(|\overline{x}|+d(\overline{x}, {\{u^*= 0\}} ))>||u_\varrho-u^* ||_{L^\infty (B_1)}$, where $\varepsilon$  is as in Lemma \ref{L6.2}.  Applying \eqref{con5}, one can arrive at $u_\varrho(\overline{x})>0$. Rescaling back to $u$, we have  $u(x_1,\frac{3}{4}\varrho)>0$. Similarly, we have $d(\underline{x}, {\{u^*> 0\}} )\ge \frac{\sqrt{3}}{2}| \frac{\sqrt{3}x_1}{3\varrho }+\frac{1}{4} |\ge \frac{\sqrt{3}}{8}$. Hence if $\varrho$ is small enough, by using \eqref{con6}, we have $\underline{x}\in {\{u_\varrho= 0\}}^\circ$.  Rescaling back to $u$, we have  $(x_1,-\frac{1}{4}\varrho)\in {\{u= 0\}}^\circ$. Finally, choosing $s_1=-\frac{1}{4}\varrho$  and $s_2=\frac{3}{4}\varrho$, we complete the proof of the claim.\\
\indent \textbf{Claim 2.} For any  $x_1\in (0,\frac{\varrho}{4})$, there exists a unique $x_2$ such that $x=(x_1,x_2)\in \Gamma_u\cap B_{\varrho}$. \\
Let $x=(x_1,x_2)\in \Gamma_u\cap B_{\varrho}\cap \{0<x_1<  \frac{\varrho}{4} \} $.
Define $x':=\frac{x}{4|x|}\in \Gamma_{u_{4|x|}}\cap\partial B_{\frac{1}{4}}\cap\{x'_1>0\}$. Suppose  $c>0$ be a small constant. Then there exists $\varrho$ small enough such that  $||u_{4|x|}-u^* ||^{\frac{1}{3}}_{C^1 (B_1)}+r\le \frac{c}{2}$, where $r=\frac{c}{4}$. With the aid of Lemma \ref{L7.4}, we have
\begin{equation*}
u_{4|x|}=0 \quad in\quad \Theta^{1}_{x'}=\{y\in B_r(x'):y=x'-te \,\ with \,\ e\in C_{\frac{1}{2}},\, t>0\},
\end{equation*}
\begin{equation*}
u_{4|x|}>0 \quad in\quad \Theta^{2}_{x'}=\{y\in B_r(x'):y=x'+te \,\ with \,\ e\in C_{\frac{1}{2}},\, t>0\}.
\end{equation*}
Rescaling back to $u$, one can see that
\begin{equation}\label{L8L1}
u=0 \quad in\quad \Theta^{1}_{x}=\{z\in B_{4|x|r}(x):z=x-4|x|te \,\ with \,\ e\in C_{\frac{1}{2}},\, t>0\},
\end{equation}
\begin{equation}\label{L8L2}
u>0 \quad in\quad \Theta^{2}_{x}=\{z\in B_{4|x|r}(x):z=x+4|x|te \,\ with \,\ e\in C_{\frac{1}{2}},\, t>0\}.
\end{equation}
For any $x_1\in (0,\frac{\varrho}{4})$, assume by contradiction that there exists a point  $(x_1, a)\in B_{\varrho}$ with $a>x_2$ such that $u(x_1, a)=0$. Taking $\underline{s}=\inf\{ a>x_2:  (x_1, a)\in \cap B_{\varrho}, u(x_1, a)=0 \}$. From \eqref{L8L2}, we have $u>0 \,\ in \,\ \Theta^{2}_{x}$. Therefore, we have  $\underline{s}>x_2$ and
\begin{equation}\label{gr1}
u(x_1,s)>0\quad\textrm{for all}\quad x_2< s<\underline{s}.
\end{equation}
Define $y:=(x_1,\underline{s})$.  Note that we also have $y\in \Gamma_u\cap B_{\varrho}$. Hence by the same argument as above shows that
\begin{equation*}
u=0 \quad in\quad \Theta^{1}_{y}=\{z\in B_{4|y|r}(y):z=y-4|y|te \,\ with \,\ e\in C_{\frac{1}{2}},\, t>0\}.
\end{equation*}
This implies that there exists $s_0>0$ satisfying  $\max\{x_2, \underline{s}-4|y|r \}<s_0<\underline{s}$ such that $u(x_1,s_0)=0$,  which contradicts to \eqref{gr1}.
Finally, we argue by contradiction once more. Fix $x_1\in (0,\frac{\varrho}{4})$ and suppose there exists a point
$(x_1, b)\in B_{\varrho}$ with $b<x_2$ such that $(x_1, b)\in \overline{\{u(x_1, \cdot)>0\}}$. Let $\overline{s}=\sup\{b<x_2: (x_1, b)\in B_{\varrho}, \, and  \,\ (x_1, b)\in  \overline{\{u(x_1, \cdot)>0\}} \}$.  Similarly, from \eqref{L8L1}, we have $\overline{s}<x_2$. Thus we have
\begin{equation}\label{gr2}
(x_1,s)\in \{u(x_1, \cdot)=0\}^\circ  \quad\textrm{for all}\quad \overline{s}<s<x_2.
\end{equation}
Define $y':=(x_1, \overline{s})\in \Gamma_u\cap B_{\varrho}$. Repeating now the same argument as the proof of \eqref{L8L2}, we have
\begin{equation*}
u>0 \quad in\quad \Theta^{2}_{y'}=\{z\in B_{4|y'|r}(y'):z=y'+4|y'|te \,\ with \,\ e\in C_{\frac{1}{2}},\, t>0\}.
\end{equation*}
Hence, there exists $r_0\in \left(\overline{s}, \min\{\overline{s}+4|y'|r, x_2 \}  \right)$ such that $u(x_1, r_0)>0$. This contradicts to \eqref{gr2}.
This completes the proof of the uniqueness of $x_2$.
\end{proof}

Now, we give an immediate corollary of Lemma \ref{L8.1}.
\begin{cor}\label{Cor8.1}
Let $u$ be a solution to \eqref{f6}. Suppose $0$ is a regular point. Then there exist a small constant $\varrho>0$ and a function $g\in C([0, \frac{\varrho}{4}))$ such that
\begin{equation*}
\Gamma(u)\cap B_{\varrho}\cap\{x: 0\le x_1<\frac{\varrho}{4} \}=\{x\in B_{\varrho} :x_2=g(x_1) \}.
\end{equation*}
\end{cor}
\begin{proof}
 Let $x=(x_1,x_2)\in \Gamma(u)\cap B_{\varrho}\cap\{x: 0\le x_1<\frac{\varrho}{4} \}$. By Lemma \ref{L8.1}, there exists  a function  $g(x)$ such that
\begin{equation*}
x_2=\left\{
\begin{aligned}
&\,g(0), \quad  \,\,\,\,  x_1=0,\\
&g(x_1), \quad  \,\,0<x_1<\frac{\varrho}{4},
\end{aligned}
\right.
\end{equation*}
where $g(0)=0$. It remains to show that $g\in C([0, \frac{\varrho}{4}))$.

Fix $y\in [0, \frac{\varrho}{4})$. Let $\sigma>0$ be small enough. According  the proof of Lemma \ref{L8.1}, we have $u(y,g(y)+\sigma)>0$ and  $u(y,g(y)-\frac{\sigma}{2})=0$. By the continuity of $u$, one can see that there exist $\delta_1, \delta_2>0$ such that $u>0$ in $B_{\delta_1}(y,g(y)+\sigma)$ and
$u=0$ in $B_{\delta_2}(y,g(y)-\frac{\sigma}{2})$. Hence, we have
\begin{equation}\label{gr3}
u(x_1,g(y)+\sigma)>0 \quad\textrm{for all}\quad |x_1-y|<\delta_1,
\end{equation}
\begin{equation}\label{gr4}
u(x_1,g(y)-\frac{\sigma}{2})=0 \quad\textrm{for all}\quad |x_1-y|<\delta_2.
\end{equation}
(Here if $y=0$ we assume $x_1\ge0$). Using  Lemma \ref{L8.1} again, it follows from \eqref{gr3} that $g(x_1)<g(y)+\sigma$.
Similarly, we derive from \eqref{gr4} that $g(x_1)\ge g(y)-\frac{\sigma}{2}>g(y)-\sigma$. Therefore, for all $\sigma>0$ sufficiently small,  we have $|g(x_1)-g(y)|<\sigma$ for $|x_1-y|<\delta$, where $\delta=\min\{\delta_1,\delta_2 \}$. This completes the proof of Corollary \ref{Cor8.1}.
\end{proof}

Now, we provide a more refined characterization for the convergence of the free boundary.
\begin{lem}\label{L8.2}
Let $u$ be any solution to \eqref{f6}.  Suppose $0$ is a regular point, and
$x\in \Gamma(u)\cap B_{\varrho}\cap\{x: 0\le x_1<\frac{\varrho}{4}\}$.
Then there exist $C>0$ such that
\begin{equation}\label{gr5}
|g(x_1)-\frac{\sqrt{3}}{3}x_1|\le C||u_{8|x_1|-u^*} ||^{\frac{1}{3}}_{L^\infty (B_1)}|x_1|.
\end{equation}
\end{lem}
\begin{proof}
 Let $x\in \Gamma(u)\cap B_{\varrho}\cap\{x: 0\le x_1<\frac{\varrho}{4}\}$. Then we have $d(x, \Gamma_{u^*})=\frac{\sqrt{3}}{2}|\frac{\sqrt{3}}{3}x_1-x_2|$. Therefore, by Lemma \ref{L6.4}, one can derive that there exists $C_1>0$ such that
\begin{equation}\label{gr6}
|\frac{\sqrt{3}}{3}x_1-x_2|\le C_1 ||u_{4|x|-u^*} ||^{\frac{1}{3}}_{L^\infty (B_1)}|x|.
\end{equation}
On the other hand, we have
\begin{equation}\label{gr7}
\begin{split}
|x|&\le |x_1|+|x_2|\\
&\le (1+\frac{\sqrt{3}}{3})|x_1|+|x_2-\frac{\sqrt{3}}{3}x_1|.
\end{split}
\end{equation}
Note that  if $\varrho>0$ is small enough, we have
\begin{equation}\label{gr8}
C_1||u_{4|x|-u^*} ||^{\frac{1}{3}}_{L^\infty (B_1)}\le \frac{1}{2}.
\end{equation}
Combining \eqref{gr6}, \eqref{gr7} and \eqref{gr8} we obtain
\begin{equation}\label{gr9}
|x_2-\frac{\sqrt{3}}{3}x_1|\le C_2||u_{4|x|-u^*} ||^{\frac{1}{3}}_{L^\infty (B_1)}|x_1|.
\end{equation}
If  $\varrho>0$ is small enough, we obtain $C_2||u_{4|x|-u^*} ||^{\frac{1}{3}}_{L^\infty (B_1)}\le 1-\frac{\sqrt{3}}{3}$. This together with \eqref{gr7} and \eqref{gr9} yields
\begin{equation}\label{gr10}
|x|\le 2|x_1|.
\end{equation}
Now, putting \eqref{gr9} and \eqref{gr10} together, we obtain
\begin{equation*}
|x_2-\frac{\sqrt{3}}{3}x_1|\le C_2||u_{8|x_1|-u^*} ||^{\frac{1}{3}}_{L^\infty (B_1)}|x_1|,
\end{equation*}
which together the fact $x_2=g(x_1)$ implies the desired conclusion.
\end{proof}

Finally, based on the normal convergence of the free boundary, we  derive the following estimate for the slope deviation.
\begin{lem}\label{L8.3}
Let $u$ be any solution to \eqref{f6}. Suppose $0$ is a regular point and $x\in \Gamma(u)\cap B_{\varrho}\cap\{x: 0\le x_1<\frac{\varrho}{4}\}$. Then $g\in C^1((0, \frac{\varrho}{4}))$ and
\begin{equation*}
|g'(x_1)-\frac{\sqrt{3}}{3}|\le C||u_{8|x_1|-u^*} ||^{\frac{1}{3}}_{L^\infty (B_1)}|x_1|.
\end{equation*}
\end{lem}
\begin{proof}
 On one hand, by Lemma \ref{L7.5}, any $x\in \Gamma_u\cap B_\varrho\cap \{ x_1>0\}$ is a usual regular point. Therefore, as known from the classical theory \cite{Pet}, $\Gamma_u$ is $C^1$ curve in a sufficiently small neighborhood of $x$.   On the other hand,  from Lemma \ref{L7.6}, if $\varrho$ is small enough, we have $e_x\to e^*$ with $e^*=(-\frac{1}{2},\frac{\sqrt{3}}{2})$. Hence $e^*$ exists with  $e^*=\frac{(-g'(x_1),1)}{\sqrt{1+(g'(x_1) )^2  }}$. By a simple calculation, one can see that $ |e_x- e^* |\ge\sqrt{C_1}|g'(x_1)-\frac{\sqrt{3}}{3}  | $ for any $C_1\in(0,3)$ provided that $\varrho$ is small enough. This together with Lemma \ref{L7.6} gives $|g'(x_1)-\frac{\sqrt{3}}{3}  |\le C_2 ||u_{4|x|}-u^* ||^{\frac{1}{3}}_{C^1 (B_1)} $. With the help of  \eqref{gr10}, one can prove the desired conclusion.
\end{proof}

\begin{proof}[Proof of Theorem \ref{TH1.4}.]
Invoke Corollary \ref{Cor8.1}, Lemma \ref{L8.2} and  Lemma \ref{L8.3}, one can deduce Theorem \ref{TH1.4}.
\end{proof}

\bibliographystyle{alpha}

\end{document}